\documentclass[article,12pt]{amsart}
\usepackage{amsfonts}
\usepackage{amsmath}
\usepackage{amssymb}
\usepackage{graphicx}
\usepackage{epsfig}
\usepackage{enumerate}
\usepackage{color}

\allowdisplaybreaks

\numberwithin{equation}{section}

\theoremstyle{plain}
\newtheorem{thm}{Theorem}[section]

\newtheorem{cor}{Corollary}[section]
\newtheorem{lem}{Lemma}[section]

\newtheorem{rem}{Remark}[section]

\newcommand{\dC}{\mathbb{C}}
\newcommand{\dE}{\mathbb{E}}
\newcommand{\dI}{\mathbb{I}}
\newcommand{\dN}{\mathbb{N}}
\newcommand{\dP}{\mathbb{P}}
\newcommand{\dR}{\mathbb{R}}

\newcommand{\dZ}{\mathbb{Z}}

\newcommand{\cD}{\mathcal{D}}
\newcommand{\cF}{\mathcal{F}}
\newcommand{\cH}{\mathcal{H}}
\newcommand{\cL}{\mathcal{L}}
\newcommand{\cM}{\mathcal{M}}
\newcommand{\cN}{\mathcal{N}}

\newcommand{\de}{\mathrm{e}}

\newcommand{\cov}{\dC\textnormal{ov}}

\newcommand{\veps}{\varepsilon}
\newcommand{\hsp}{\hspace{0.5cm}}
\newcommand{\hspp}{\hspace{0.3cm}}
\newcommand{\hspand}{\hsp\text{and}\hsp}
\newcommand{\wh}{\widehat}
\newcommand{\wt}{\widetilde}

\newcommand{\cvgas}{~ \overset{\textnormal{a.s.}}{\longrightarrow} ~}
\newcommand{\cvgp}{~ \overset{\dP}{\longrightarrow} ~}
\newcommand{\cvgl}{~ \overset{\cD}{\longrightarrow} ~}

\email{frederic.proia@univ-angers.fr}
\email{marius.soltane.etu@univ-lemans.fr}

\keywords{RCAR process, MA process, Random coefficients, Least squares estimation, Stationarity, Ergodicity, Asymptotic normality, Autocorrelation.}

\begin{document}

\title[Correlation in the random coefficients]
{A test of correlation in the random coefficients of an autoregressive process
\vspace{2ex}}
\author[F. Pro\"ia]{Fr\'ed\'eric Pro\"ia}
\address{Laboratoire Angevin de REcherche en MAth\'ematiques (LAREMA), CNRS, Universit\'e d'Angers, Universit\'e Bretagne Loire. 2 Boulevard Lavoisier, 49045 Angers cedex 01, France.}
\author[M. Soltane]{Marius Soltane}
\address{Laboratoire Manceau de Math\'ematiques, Le Mans Universit\'e, Avenue O. Messiaen, 72085 Le Mans cedex 9, France.}

\thanks{}

\begin{abstract}
A random coefficient autoregressive process is deeply investigated in which the coefficients are correlated. First we look at the existence of a strictly stationary causal solution, we give the second-order stationarity conditions and the autocorrelation function of the process. Then we study some asymptotic properties of the empirical mean and the usual estimators of the process, such as convergence, asymptotic normality and rates of convergence, supplied with the appropriate assumptions on the driving perturbations. Our objective is to get an overview of the influence of correlated coefficients in the estimation step, through a simple model. In particular, the lack of consistency is shown for the estimation of the autoregressive parameter when the independence hypothesis is violated in the random coefficients. Finally, a consistent estimation is given together with a testing procedure for the existence of correlation in the coefficients. While convergence properties rely on the ergodicity, we use a martingale approach to reach most of the results.
\end{abstract}

\maketitle

\noindent \textbf{Notations and conventions.} In the whole paper, $I_{p}$ is the identity matrix of order $p$, $[v]_{i}$ refers to the $i$--th element of any vector $v$ and $M_{i}$ to the $i$--th column of any matrix $M$. In addition, $\rho(M)$ is the spectral radius of any square matrix $M$, $M \circ N$ is the Hadamard product between matrices $M$ and $N$, and $\ln^{+}\! x = \max(\ln x, 0)$. We make the conventions $\sum_{\varnothing} = 0$ and $\prod_{\varnothing} = 1$. Symbols $o(\cdot)$ and $O(\cdot)$ with regard to random sequences will be repeatedly used in the same way as applied to real-valued functions: as $n \rightarrow +\infty$, for some positive deterministic rate $(v_{n})$, $X_{n} = o(v_{n})$ a.s. means that $X_{n}/v_{n}$ converges almost surely to 0 whereas $X_{n} = O(v_{n})$ a.s. means, in the terminology of \cite{Duflo97}, that for almost all $\omega$, $X_{n}(\omega) = O(v_{n})$, that is $\vert X_{n}(\omega) \vert \leq C(\omega)\, v_{n}$ for some finite $C(\omega) \geq 0$ and $n \geq N(\omega)$.

%%%%%%%%%%%%%%%%%%%%%%%%%%%%%%%%%%%%%%%%%%%%%%%%%%%%%%%%%%%%%%%%%%%%%%%%%%%%%%%%
\section{Introduction and Motivations}
\label{SecIntro}

In the econometric field, nonlinear time series are now very popular. Our interest lies in some kind of generalization of the standard first-order autoregressive process through random coefficients. The well-known random coefficient autoregressive process RCAR(1) is defined for $t \in \dZ$ by
$$
X_{t} = (\theta + \eta_{t}) X_{t-1} + \veps_{t}
$$
where $(\veps_{t})$ and $(\eta_{t})$ are uncorrelated white noises. Since the seminal works of And\v{e}l \cite{Andel76} and Nicholls and Quinn \cite{NichollsQuinn81a}, stationarity conditions for such processes have been widely studied under various assumptions on the moments of $(\veps_{t})$ and $(\eta_{t})$. Namely, the process was proven to be second-order stationary if $\theta^2 + \tau_2 < 1$ where $\tau_2$ stands for the variance of $(\eta_{t})$. Quite recently, Aue \textit{et al.} \cite{AueHorvathSteinebach06} have given necessary and sufficient conditions for the existence and uniqueness of a strictly stationary solution of the RCAR(1) process, derived from the more general paper of Brandt \cite{Brandt86}, and some of our technical assumptions are inspired by their works. However, the flexibility induced by RCAR processes is balanced by the absence of correlation between two consecutive values of the random coefficient. In a time series context, this seems somehow counterintuitive and difficult to argue. Our main objective is precisely to show that the violation of the independence hypothesis in the coefficients, though quite likely for a stochastic phenomenon, leads to a falsification of the whole estimation procedures, and therefore of statistical interpretations. That is the reason why we suggest in this paper an example of random coefficients having a short (finite) memory, in the form of a moving-average dynamic, for which the estimation of the mean value shall be conducted as if they were uncorrelated. For all $t \in \dZ$, we consider the first-order autoregressive process given by
\begin{equation}
\label{XAR}
X_{t} = \theta_{t}\, X_{t-1} + \veps_{t}
\end{equation}
where $\theta_{t}$ is a random coefficient generated by the moving-average structure
\begin{equation}
\label{MACoef}
\theta_{t} = \theta + \alpha\, \eta_{t-1} + \eta_{t}.
\end{equation}
This choice of dependence pattern in the coefficients is motivated by Prop. 3.2.1 of \cite{BrockwellDavis96} which states that any stationary process having finite memory is solution of a moving-average structure. In other words, there exists a white noise such that the random coefficients admit the decomposition given above, and this justifies our interest in \eqref{MACoef}. We can find the foundations of a similar model in Koubkov\`a \cite{Koubkova82} or in a far more general way in Brandt \cite{Brandt86}, but as we will see throughout the paper our objectives clearly diverge. While their works concentrate on the properties of the stationary solution, a large part of this paper focuses on inference. The set of hypotheses that we retain is presented at the end of this introduction, and Section \ref{SecProp} is devoted to the existence, the uniqueness and the stationarity conditions of $(X_{t})$. This preliminary study enables us to derive the autocorrelation function of the process. In Section \ref{SecOLS}, the empirical mean of the process and the usual estimators of $\theta$ and $\sigma_2$ are investigated, where $\sigma_2$ stands for the variance of $(\veps_{t})$. In particular, we establish some almost sure convergences, asymptotic normalities and rates of convergence, and we also need some results on the fourth-order moments of the process that we deeply examinate. The surprising corollary of these calculations is that the estimation is not consistent for $\theta$ as soon as $\alpha \neq 0$, whereas it is well-known that consistency is preserved in the RCAR(1) process. That leads us in Section \ref{SecTest} to build a consistent estimation together with its asymptotic normality, and to derive a statistical procedure for the existence of correlation in the coefficients. In Section \ref{SecProofs}, we finally prove our results. The estimation of RCAR processes has also been widely addressed in the stationary case, for example by Nicholls and Quinn \cite{NichollsQuinn81b} and later by Schick \cite{Schick96}, using either least squares or quasi-maximum likelihood. The crucial point in these works is the strong consistency of the estimation, whereas it appears in our results that the introduction of correlation in the coefficients is only possible at the cost of consistency. In a general way, our objective is to get an overview of the influence of correlated coefficients in the estimation step through a simple model, to open up new perspectives for more complex structures of dependence. Throughout the paper, we will recall the well-known results related to the first-order stationary RCAR process that are supposed to match with ours for $\alpha=0$. The reader may find a whole survey in Nicholls and Quinn \cite{NichollsQuinn82} and without completeness, we also mention the investigations of \cite{Robinson78}, \cite{Jurgens85}, \cite{HwangBasawa05}, \cite{HwangBasawaKim06}, \cite{BerkesHorvathLing09} about inference on RCAR processes, or the unified procedure of Aue and Horv\'{a}th \cite{AueHorvath11} and references inside. For all $a > 0$, we note the moments
$$
\sigma_{a} = \dE[\veps_0^{\, a}] \hspand \tau_{a} = \dE[\eta_0^{\, a}].
$$
To simplify the calculations, we consider the family of vectors given by
\begin{equation}
\label{FamVecMom2}
U_0 = \begin{pmatrix}
1 \\ 0 \\ \tau_2
\end{pmatrix}, \hsp
U_1 = \begin{pmatrix}
0 \\ \tau_2 \\ 0
\end{pmatrix}, \hsp
U_2 = \begin{pmatrix}
\tau_2 \\ 0 \\ \tau_4
\end{pmatrix}.
\end{equation}
A particular $3 \times 3$ matrix is used all along the study to characterize the second-order properties of the process, it is based on $\{ U_0, U_1, U_2 \}$ in such a way that
\begin{equation}
\label{M}
M =
\begin{pmatrix}
\theta^2 + \tau_2 & 2\, \alpha\, \theta & \alpha^2 \\
2\, \theta\, \tau_2 &  2\, \alpha\, \tau_2 & 0 \\
\theta^2\, \tau_2 + \tau_4 & 2\, \alpha\, \theta\, \tau_2 & \alpha^2\, \tau_2
\end{pmatrix} \hsp \text{with} \hsp \left\{
\begin{array}{l}
M_1 = \theta^2\, U_0 + 2\, \theta\, U_1 + U_2 \\
M_2 = 2\, \alpha\, ( \theta\, U_0 + U_1) \\
M_3 = \alpha^2\, U_0.
\end{array}
\right.
\end{equation}
Similarly, the fourth-order properties of the process rest upon the family of vectors $\{ V_0, \hdots, V_4 \}$ where
\begin{equation}
\label{FamVecMom4}
V_0 = \begin{pmatrix}
1 \\ 0 \\ \tau_2 \\ 0 \\ \tau_4
\end{pmatrix}, \hsp
V_1 = \begin{pmatrix}
0 \\ \tau_2 \\ 0 \\ \tau_4 \\ 0
\end{pmatrix}, \hsp
V_2 = \begin{pmatrix}
\tau_2 \\ 0 \\ \tau_4 \\ 0 \\ \tau_6
\end{pmatrix}, \hsp
V_3 = \begin{pmatrix}
0 \\ \tau_4 \\ 0 \\ \tau_6 \\ 0
\end{pmatrix}, \hsp
V_4 = \begin{pmatrix}
\tau_4 \\ 0 \\ \tau_6 \\ 0 \\ \tau_8
\end{pmatrix}.
\end{equation}
There are used to build the $5 \times 5$ matrix $H$ whose columns are defined as
\begin{equation}
\label{H}
\left\{
\begin{array}{l}
H_1 = \theta^4\, V_0 + 4\, \theta^3\, V_1 + 6\, \theta^2\, V_2 + 4\, \theta\, V_3 + V_4 \\
H_2 = 4\, \alpha\, ( \theta^3\, V_0 + 3\, \theta^2\, V_1 + 3\, \theta\, V_2 + V_3 ) \\
H_3 = 6\, \alpha^2\, ( \theta^2\, V_0 + 2\, \theta\, V_1 + V_2 ) \\
H_4 = 4\, \alpha^3\, ( \theta\, V_0 + V_1 ) \\
H_5 = \alpha^4\, V_0.
\end{array}
\right.
\end{equation}
Explicitly,
$$
H = \begin{pmatrix}
\theta^4 + 6\, \theta^2\, \tau_2 + \tau_4 & 4\, \alpha\,( \theta^3 + 3\, \theta\, \tau_2) & 6\, \alpha^2\, ( \theta^2 + \tau_2) & 4\, \alpha^3\, \theta & \alpha^4 \\
4\, \theta^3\, \tau_2 + 4\, \theta\, \tau_4 & 4\, \alpha\, (3\, \theta^2\, \tau_2 + \tau_4) & 12\, \alpha^2\, \theta\, \tau_2 & 4\, \alpha^3\, \tau_2 & 0 \\
\theta^4\, \tau_2 + 6\, \theta^2\, \tau_4 + \tau_6 & 4\, \alpha\, (\theta^3\, \tau_2 + 3\, \theta\, \tau_4) & 6\, \alpha^2\, (\theta^2\, \tau_2 + \tau_4) & 4\, \alpha^3\, \theta\, \tau_2 & \alpha^4\, \tau_2 \\
4\, \theta^3\, \tau_4 + 4\, \theta\, \tau_6 & 4\, \alpha\, (3\, \theta^2\, \tau_4 + \tau_6) & 12\, \alpha^2\, \theta\, \tau_4 & 4\, \alpha^3\, \tau_4 & 0 \\
\theta^4\, \tau_4 + 6\, \theta^2\, \tau_6 + \tau_8 & 4\, \alpha\, (\theta^3\, \tau_4 + 3\, \theta\, \tau_6) & 6\, \alpha^2\, (\theta^2\, \tau_4 + \tau_6) & 4\, \alpha^3\, \theta\, \tau_4 & \alpha^4\, \tau_4
\end{pmatrix}.
$$
Various hypotheses on the parameters will be required (not always simultaneously) throughout the study, closely related to the distribution of the perturbations.
\begin{enumerate}[(H$_1$)]
\item The processes $(\veps_{t})$ and $(\eta_{t})$ are mutually independent strong white noises such that $\dE[\ln^{+}\!\vert \veps_0 \vert] < \infty$ and $\dE[\ln \vert \theta + \alpha\, \eta_0 + \eta_1 \vert] < 0$.
\item $\sigma_{2 k +1} = \tau_{2k+1} = 0$ for any $k \in \dN$ such that the moments exist.
\item $\sigma_2 > 0$, $\tau_2 > 0$, $\sigma_2 < \infty$, $\tau_4 < \infty$ and $\rho(M) < 1$.
\item $\sigma_4 < \infty$, $\tau_8 < \infty$ and $\rho(H) < 1$.
\item There exists continuous mappings $g$ and $h$ such that $\sigma_4 = g(\sigma_2)$ and $\tau_4 = h(\tau_2)$.
\end{enumerate}

\begin{rem}
Clearly, (H$_2$) can be replaced by the far less restrictive natural condition $\sigma_1 = \tau_1 = 0$. Considering that all existing odd moments of $(\veps_{t})$ and $(\eta_{t})$ are zero is only a matter of simplification of the calculations, that are already quite tricky to conduct. An even more general (and possible) study must include the contributions of $\sigma_3$, $\tau_3$, $\tau_5$ and $\tau_7$ in the whole calculations.
\end{rem}

\begin{rem}
(H$_5$) is satisfied in the centered Gaussian case with $g(t) = h(t) = 3\, t^2$. It is also satisfied for most of the distributions used to drive the noise of regression models (centered uniform, Student, Laplace, etc.). Nevertheless, it is a strong assumption only used at the end of the study.
\end{rem}

Short explanations of the remarks appearing in Sections \ref{SecProp} and \ref{SecOLS} are given at the beginning of Section \ref{SecProofs}.
%%%%%%%%%%%%%%%%%%%%%%%%%%%%%%%%%%%%%%%%%%%%%%%%%%%%%%%%%%%%%%%%%%%%%%%%%%%%%%%%

%%%%%%%%%%%%%%%%%%%%%%%%%%%%%%%%%%%%%%%%%%%%%%%%%%%%%%%%%%%%%%%%%%%%%%%%%%%%%%%%
\section{Stationarity and Autocorrelation}
\label{SecProp}

It is well-known and easy to establish that the sequence of coefficients $(\theta_{t})$ given by \eqref{MACoef} is a strictly stationary and ergodic process with mean $\theta$ and autocovariance function given by
$$
\gamma_{\theta}(0) = \tau_2\, (1 + \alpha^2), \hsp \gamma_{\theta}(1) = \alpha\, \tau_2 \hspand \gamma_{\theta}(h) = 0 \hsp (\vert h \vert > 1).
$$
Clearly, any solution of \eqref{XAR} satisfies a recurrence equation, and the first result to investigate is related to the existence of a causal, strictly stationary and ergodic solution.
\begin{thm}
\label{ThmCausal}
Assume that (H$_1$) holds. Then almost surely, for all $t \in \dZ$,
\begin{equation}
\label{XCausal}
X_{t} = \veps_{t} + \sum_{k=1}^{\infty} \veps_{t-k}\, \prod_{\ell=0}^{k-1} (\theta + \alpha\, \eta_{t-\ell-1} + \eta_{t-\ell}).
\end{equation}
In addition, $(X_{t})$ is strictly stationary and ergodic.
\end{thm}
\begin{proof}
See Section \ref{SecProofThmCausal}.
\end{proof}

By extension, the same kind of conclusions may be obtained on any process $(\veps_{t}^{\, a}\, \eta_{t}^{\, b}\, X_{t}^{\, c})$ for $a,b,c \geq 0$, assuming suitable conditions of moments. As a corollary, it will be sufficient to work on $\dE[\veps_{t}^{\, a}\, \eta_{t}^{\, b}\, X_{t}^{\, c}]$ in order to identify the asymptotic behavior (for $n \rightarrow \infty$) of empirical moments like
$$
\frac{1}{n} \sum_{t=1}^{n} \veps_{t}^{\, a}\, \eta_{t}^{\, b}\, X_{t}^{\, c}.
$$
According to the causal representation of the above theorem, the process is adapted to the filtration defined as
\begin{equation}
\label{Filtr}
\cF_{t} = \sigma( (\veps_{s}, \eta_{s}),\, s \leq t).
\end{equation}
We are now interested in the existence of the second-order properties of the process, under some additional hypotheses. We derive below its autocorrelation function using the previous notations and letting
\begin{equation}
\label{N}
N = \begin{pmatrix}
\theta & \alpha & 0 \\
\tau_2 & 0 & 0 \\
\theta\, \tau_2 & \alpha\, \tau_2 & 0
\end{pmatrix}
\hsp \text{with} \hsp \left\{
\begin{array}{l}
N_1 = \theta\, U_0 + U_1 \\
N_2 = \alpha\, U_0 \\
N_3 = 0,
\end{array}
\right.
\end{equation}
and we take advantage of the calculations to guarantee the unicity of the second-order stationary solution.

\begin{thm}
\label{ThmStat}
Assume that (H$_1$)--(H$_3$) hold. Then, $(X_{t})$ is a strictly and second-order stationary process with mean zero and autocovariance function given by
\begin{equation}
\label{ACV}
\gamma_{X}(h) = \sigma_2\, \big[ N^{\vert\, h\, \vert}\, (I_3-M)^{-1}\, U_0 \big]_1
\end{equation}
for $h \in \dZ$. Its autocorrelation function is defined as
\begin{equation}
\label{ACF}
\rho_{X}(h) = \frac{\gamma_{X}(h)}{\gamma_{X}(0)}.
\end{equation}
In addition, this is the unique causal ergodic strictly and second-order stationary solution.
\end{thm}
\begin{proof}
See Section \ref{SecProofThmStat}.
\end{proof}

\begin{rem}
\label{RemPartCase}
Suppose that the process is stationary with second-order moments such that the parameters satisfy $2\, \alpha\, \tau_2 = 1$. Then, \eqref{ACV} leads to $\gamma_{X}(0)=0$, meaning that $(X_{t})$ is a deterministic process. This case is naturally excluded from the study, just like $\sigma_2 = 0$ leading to the same conclusion.
\end{rem}

\begin{rem}
For $\alpha=0$, the set of eigenvalues of $M$ is $\{ \theta^2 + \tau_2, 0, 0 \}$. Thus, the assumption $\rho(M) < 1$ reduces to $\theta^2 + \tau_2 < 1$, which is a well-known result for the stationarity of RCAR(1) processes.
\end{rem}
%%%%%%%%%%%%%%%%%%%%%%%%%%%%%%%%%%%%%%%%%%%%%%%%%%%%%%%%%%%%%%%%%%%%%%%%%%%%%%%%

%%%%%%%%%%%%%%%%%%%%%%%%%%%%%%%%%%%%%%%%%%%%%%%%%%%%%%%%%%%%%%%%%%%%%%%%%%%%%%%%
\section{Empirical mean and Usual estimation}
\label{SecOLS}

Assume that a time series $(X_{t})$ generated by \eqref{XAR}--\eqref{MACoef} is observable on the interval $t \in \{ 0, \hdots, n \}$, for $n \geq 1$. We additionally suppose that $X_0$ has the strictly stationary and ergodic distribution of the process.
\begin{rem}
\label{RemInitCond}
Making the assumption that $X_0$ has the strictly stationary and ergodic distribution of the process is only a matter of simplification of the calculations. To be complete, assume that $(Y_{t})$ is generated by the same recurrence with initial value $Y_0$. Then for all $t \geq 1$,
\begin{equation*}
X_{t}-Y_{t} = (X_0-Y_0)\, \prod_{\ell=1}^{t} (\theta + \alpha\, \eta_{\ell-1} + \eta_{\ell}).
\end{equation*}
For a sufficiently large $t$ and letting $\kappa = \dE[\ln \vert \theta + \alpha\, \eta_0 + \eta_1 \vert] < 0$, it can be shown (see Section \ref{SecProofRem} for details) that, almost surely,
\begin{equation*}
\vert X_{t}-Y_{t} \vert \leq \vert X_0-Y_0 \vert\, \de^{\frac{\kappa\, t}{2}}.
\end{equation*}
Then $Y_0$ could by any random variable satisfying $\vert X_0-Y_0 \vert < \infty$ a.s. and having at least as many moments as $X_0$.
\end{rem}
Denote the sample mean by
\begin{equation}
\label{EmpMean}
\bar{X}_{n} = \frac{1}{n}\, \sum_{t=1}^{n} X_{t}.
\end{equation}
Then, we have the following result, where the asymptotic variance $\kappa^2$ will be explicitly given in \eqref{VarEmpMean}.
\begin{thm}
\label{ThmEmpMean}
Assume that (H$_1$)--(H$_2$) hold. Then as $n$ tends to infinity, we have the almost sure convergence
\begin{equation}
\label{CvgMean}
\bar{X}_{n} \cvgas 0.
\end{equation}
In addition, if (H$_3$) also holds, we have the asymptotic normality
\begin{equation}
\label{TlcMean}
\sqrt{n}\, \bar{X}_{n} \cvgl \cN(0, \kappa^2).
\end{equation}
\end{thm}
\begin{proof}
See Section \ref{SecProofThmEmpMean}.
\end{proof}

\begin{rem}
\label{RemPartCaseEmpMean}
For $\alpha=0$, our calculations lead to 
\begin{equation}
\label{VarKappa0}
\kappa^2_{0} = \frac{\sigma_2\, (1 - \theta^2)}{(1-\theta)^2 (1 - \theta^2 - \tau_2)}.
\end{equation}
If in addition $\tau_2 = 0$, we find that
\begin{equation}
\label{VarKappa00}
\kappa^2_{00} = \frac{\sigma_2}{(1-\theta)^2}
\end{equation}
which is a result that can be deduced from Thm. 7.1.2 of \cite{BrockwellDavis96}.
\end{rem}

Now, consider the estimator given by
\begin{equation}
\label{EstOLS}
\wh{\theta}_{n} = \frac{\sum_{t=1}^{n} X_{t-1} X_{t}}{\sum_{t=1}^{n} X_{t-1}^{\, 2}}.
\end{equation}
It is essential to be well aware that $\wh{\theta}_{n}$ is \textit{not} the OLS estimate of $\theta$ as soon as $\alpha \neq 0$. This choice of estimate is a consequence of our objectives : to show that an OLS estimation of $\theta$ in a standard RCAR(1) model may lead to inappropriate conclusions (due to correlation in the coefficients). Indeed, we shall see in this section that it is not consistent for $\alpha \neq 0$, and we will provide its limiting value. We will also establish that it remains asymptotically normal. This estimator will be described as the \textit{usual} one afterwards. Denote by
\begin{equation}
\label{LimT}
\theta^{*} = \frac{\theta}{1 - 2\, \alpha\, \tau_2}
\end{equation}
and recall that $2\, \alpha\, \tau_2 \neq 1$. The asymptotic variance $\omega^2$ in the central limit theorem will be built step by step in Section \ref{SecProofThmOLS} and given in \eqref{VarOmega}.

\begin{thm}
\label{ThmOLS}
Assume that (H$_1$)--(H$_3$) hold. Then as $n$ tends to infinity, we have the almost sure convergence
\begin{equation}
\label{CvgOLS}
\wh{\theta}_{n} \cvgas \theta^{*}.
\end{equation}
In addition, if (H$_4$) holds, we have the asymptotic normality
\begin{equation}
\label{TlcOLS}
\sqrt{n}\, \big( \wh{\theta}_{n} - \theta^{*} \big) \cvgl \cN(0, \omega^2).
\end{equation}
\end{thm}
\begin{proof}
See Section \ref{SecProofThmOLS}.
\end{proof}

\begin{rem}
\label{RemPartCaseOLS}
For $\alpha=0$, $\theta^{*} = \theta$ and, as it is well-known, the estimation is consistent for $\theta$. In addition, the coefficients matrix $K$ defined in \eqref{K} takes the very simplified form where each term is zero except $K_{11} = \sigma_2$ and $K_{22} = \tau_2$. Similarly, only the first columns of $M$ and $H$ are nonzero. Then, letting $\lambda_0 = \dE[X_{t}^2] = \gamma_{X}(0)$ and $\delta_0 = \dE[X_{t}^4]$ as in the associated proof, the asymptotic variance is now
$$
\omega^2_0 = \frac{\sigma_2}{\lambda_0} + \frac{\tau_2\, \delta_0}{\lambda_0^2}.
$$
One can check that this is a result of Thm. 4.1 in \cite{NichollsQuinn81b}, in the particular case of the RCAR(1) process but under more natural hypotheses (they assume that $\dE[X_{t}^4] < \infty$ while we derive it from some moments conditions on the noises). Explicitly, it is given by
\begin{equation}
\label{VarOmega0}
\omega^2_0 = \frac{(1 - \theta^2 - \tau_2)\, (\tau_2\, \sigma_4\, (\theta^2 + \tau_2 - 1) + \sigma_2^2\, (\theta^4 + \tau_4 - 6\, \tau_2^2 - 1) )}{\sigma_2^2\, (\theta^4 + \tau_4 + 6\, \theta^2\, \tau_2 - 1) }.
\end{equation}
If in addition $\tau_2 = \tau_4 = 0$, we find that
\begin{equation}
\label{VarOmega00}
\omega^2_{00} = 1-\theta^2
\end{equation}
which is a result stated in Prop. 8.10.1 of \cite{BrockwellDavis96}, for example.
\end{rem}

\begin{rem}
\label{RemMom4H0}
For $\alpha=0$, the set of eigenvalues of $H$ is $\{ \theta^4 + 6\, \theta^2\, \tau_2 + \tau_4, 0, 0, 0, 0 \}$. Thus, the assumption $\rho(H) < 1$ reduces to $\theta^4 + 6\, \theta^2\, \tau_2 + \tau_4 < 1$, which may be seen as a condition of existence of fourth-order moments for the RCAR(1) process.
\end{rem}

\begin{thm}
\label{ThmOLSRat}
Assume that (H$_1$)--(H$_4$) hold. Then as $n$ tends to infinity, we have the rates of convergence
\begin{equation}
\label{LfqOLS}
\frac{1}{\ln n} \sum_{t=1}^{n} \big( \wh{\theta}_{t} - \theta^{*} \big)^2 \cvgas \omega^2
\end{equation}
and
\begin{equation}
\label{LilOLS}
\limsup_{n\, \rightarrow\, +\infty}~ \frac{n}{2\, \ln \ln n} \big( \wh{\theta}_{n} - \theta^{*} \big)^2 = \omega^2 \hspp \textnormal{a.s.}
\end{equation}
\end{thm}
\begin{proof}
See Section \ref{SecProofThmOLSRat}.
\end{proof}

\begin{rem}
The above theorem leads to the usual rate of convergence for the estimation of parameters driving stable models,
\begin{equation}
\label{RatOLS}
\big( \wh{\theta}_{n} - \theta^{*} \big)^2 = O\!\left( \frac{\ln \ln n}{n} \right) \hspp \textnormal{a.s.}
\end{equation}
\end{rem}

\begin{rem}
Even if it is of reduced statistical interest, the same rates of convergence may be reached for $\bar{X}_{n}$.
\end{rem}

Finally we build the residual set given, for all $1 \leq t \leq n$, by
\begin{equation}
\label{ResSet}
\wh{\veps}_{t} = X_{t} - \wh{\theta}_{n}\, X_{t-1}.
\end{equation}
The usual estimator of $\sigma_2$ is defined as
\begin{equation}
\label{EstOLSVar}
\wh{\sigma}_{2, n} = \frac{1}{n}\, \sum_{t=1}^{n} \wh{\veps}_{t}^{~ 2}.
\end{equation}
Denote by
\begin{equation}
\label{LimS2}
\sigma_2^{*} = \big( 1 - (\theta^{*})^2 \big)\, \gamma_{X}(0).
\end{equation}

\begin{thm}
\label{ThmCvgS2}
Assume that (H$_1$)--(H$_3$) hold. Then as $n$ tends to infinity, we have the almost sure convergence
\begin{equation}
\label{CvgOLS}
\wh{\sigma}_{2, n} \cvgas \sigma_2^{*}.
\end{equation}
\end{thm}
\begin{proof}
By ergodicity, the development of $\wh{\sigma}_{2, n}$ in \eqref{EstOLSVar} leads to
\begin{equation*}
\wh{\sigma}_{2, n} \cvgas \big( 1 + (\theta^{*})^2 \big)\, \gamma_{X}(0) - 2\, \theta^{*}\, \gamma_{X}(1).
\end{equation*}
But the definition of $\wh{\theta}_{n}$ in \eqref{EstOLS} also implies $\gamma_{X}(1) = \theta^{*}\, \gamma_{X}(0)$, leading to $\sigma_2^{*}$.
\end{proof}

\begin{rem}
For $\alpha=0$, \eqref{LimS2} becomes
\begin{equation}
\label{LimS20}
\sigma_{2,0}^{*} = \frac{\sigma_2\, (1 - \theta^2)}{1 - \theta^2 - \tau_2}.
\end{equation}
In their work, Nicholls and Quinn \cite{NichollsQuinn81b} have taken into consideration the fact that this estimator of $\sigma_2$ was not consistent, that is the reason why they suggested a modified estimator that we will take up in the next section. Now if $\tau_2=0$, we reach the well-known consistency.
\end{rem}

%%%%%%%%%%%%%%%%%%%%%%%%%%%%%%%%%%%%%%%%%%%%%%%%%%%%%%%%%%%%%%%%%%%%%%%%%%%%%%%%

%%%%%%%%%%%%%%%%%%%%%%%%%%%%%%%%%%%%%%%%%%%%%%%%%%%%%%%%%%%%%%%%%%%%%%%%%%%%%%%%
\section{A test for correlation in the coefficients}
\label{SecTest}

We now apply a Yule-Walker approach up to the second-order autocorrelation. Using the notations of Theorem \ref{ThmStat} and letting $\gamma = \alpha\, \tau_2$,
\begin{equation*}
\left\{
\begin{array}{lcl}
(1 - 2\, \rho_{X}^2(1))\, \theta & = & (1 - 2\, \rho_{X}(2))\, \rho_{X}(1) \\
(1 - 2\, \rho_{X}^2(1))\, \gamma & = & \rho_{X}(2) - \rho_{X}^2(1).
\end{array}
\right.
\end{equation*}
By ergodicity, a consistent estimation of $\theta^{*} = \rho_{X}(1)$ and $\vartheta^{*} = \rho_{X}(2)$ is achieved \textit{via}
\begin{equation}
\label{OLSEst2}
\wh{\theta}_{n} = \frac{\sum_{t=1}^{n} X_{t-1} X_{t}}{\sum_{t=1}^{n} X_{t-1}^{\, 2}} \hsp \text{and} \hsp \wh{\vartheta}_{n} = \frac{\sum_{t=2}^{n} X_{t-2} X_{t}}{\sum_{t=2}^{n} X_{t-2}^{\, 2}}
\end{equation}
respectively. We define the mapping from $[-1\,;\,1]\backslash\{\pm \frac{1}{\sqrt{2}}\} \times [-1\,;\,1]$ to $\dR^2$ as
\begin{equation}
\label{MapDelta}
f : (x,y) \mapsto \left( \frac{(1-2y)x}{1 - 2x^2}, ~ \frac{y-x^2}{1 - 2x^2} \right)
\end{equation}
and the new couple of estimates
\begin{equation}
\label{NewEst}
( \wt{\theta}_{n}, ~ \wt{\gamma}_{n} ) = f( \wh{\theta}_{n}, \wh{\vartheta}_{n}).
\end{equation}
To be consistent with \eqref{MapDelta}, we assume in the sequel that $\sqrt{2}\, \theta \neq \pm (1-2\, \alpha\, \tau_2)$. We also assume that $\psi^{\, 0}_{\, 0} \neq 0$, where $\psi^{\, 0}_{\, 0}$ is described below. Since it seems far too complicated, we do not give any reduced form to the latter hypothesis, instead we gather in $\Theta^{*} = \{ \sqrt{2}\, \theta = \pm (1-2\, \alpha\, \tau_2) \} \cup \{ \psi^{\, 0}_{\, 0} = 0 \}$ the pathological cases and we pick the parameters outside $\Theta^{*}$ to conclude our study. It obviously follows that $\wt{\theta}_{n} \overset{\textnormal{a.s.}}{\longrightarrow} \theta$ and $\wt{\gamma}_{n} \overset{\textnormal{a.s.}}{\longrightarrow} \gamma$. In the following theorem, we establish the asymptotic normality of these new estimates, useful for the testing procedure. We denote by $\nabla f$ the Jacobian matrix of $f$.

\begin{thm}
\label{ThmTlcNewEst}
Assume that (H$_1$)--(H$_4$) hold. Then as $n$ tends to infinity, we have the asymptotic normality
\begin{equation}
\label{TlcNewEst}
\sqrt{n}\, \begin{pmatrix}
\wt{\theta}_{n} - \theta \\
\wt{\gamma}_{n} - \gamma
\end{pmatrix} \cvgl \cN(0, \Psi)
\end{equation}
where $\Sigma$ is a covariance given in \eqref{Sig} and
\begin{equation}
\label{Psi}
\Psi = \nabla^{\, T} f(\theta^{*}, \vartheta^{*})\, \Sigma\, \nabla f(\theta^{*}, \vartheta^{*}).
\end{equation}
\end{thm}
\begin{proof}
See Section \ref{SecProofThmTlcNewEst}.
\end{proof}
Assuming random coefficients (that is, $\tau_2 > 0$), note that $\gamma = 0 \Leftrightarrow \alpha=0$. Our last objective is to build a testing procedure for
\begin{equation}
\label{TestCorr}
\cH_0\,:\,``\alpha = 0" \hsp \text{vs} \hsp \cH_1\,:\,``\alpha \neq 0".
\end{equation}
As it is explained in Remark \ref{RemSig}, despite its complex structure, $\Psi$ only depends on the parameters. Let $\psi = \psi(\theta, \alpha, \{ \tau_{k} \}_{2,4,6,8}, \{ \sigma_{\ell} \}_{2,4})$ be the the lower right element of $\Psi$, and $\psi^{\, 0} = \psi(\theta, 0, \{ \tau_{k} \}_{2,4,6,8}, \{ \sigma_{\ell} \}_{2,4})$. The explicit calculation under $\cH_0$ gives $\theta^{*} = \theta$, $\vartheta^{*} = \theta^2$ and
\begin{equation}
\psi^{\, 0} = \frac{\psi^{\, 0}_{\, 0}}{(1 - 2\, \theta^2)^2\, \sigma_2^2\, (\theta^4 + 6\, \theta^2\, \tau_2 + \tau_4 - 1)}
\end{equation}
where the numerator is given by
\begin{eqnarray*}
\psi^{\, 0}_{\, 0} & = & (\tau_2 + \theta^2 - 1)\, \big[ \sigma_4\, \tau_2\, ( (6\, \theta^2  - 1)\, \tau_2^2 + ( 8\, \theta^4 - 9\, \theta^2 + 1)\, \tau_2 \\
 & & \hsp \hsp \hsp + ~ 2\, \theta^2\, (\theta^2 - 1)^2) + \sigma_2^2\, \tau_2\, (-36\, \tau_2^2\, \theta^2 + 6\, \tau_2^2 - 12\, \tau_2\, \theta^4 \\
 & & \hsp \hsp \hsp + ~ 12\, \tau_2\, \theta^2 - 6\, \theta^6 + 17\, \theta^4 + 6\, \tau_4\, \theta^2 - 12\, \theta^2 - \tau_4 + 1) \\ \nonumber
 & & \hsp \hsp \hsp + ~ \sigma_2^2\, (\theta^6 - \theta^4 + \theta^2\, \tau_4 - \theta^2 - \tau_4 + 1) \big]
\end{eqnarray*}
and assumed to be nonzero (by excluding $\Theta^{*}$). As a corollary, $\psi^{\, 0}$ continuously depends on the parameters under our additional hypothesis (see Remark \ref{RemMom4H0}). Suppose also that (H$_5$) holds, so that $\psi^{\, 0} = \psi^{\, 0}(\theta, \tau_2, \sigma_2)$, and consider
\begin{equation*}
\wh{\psi}_{n}^{\, 0} = \psi^{\, 0}(\bar{\theta}_{n}, \bar{\tau}_{2,n}, \bar{\sigma}_{2,n})
\end{equation*}
where $\bar{\theta}_{n}$ is either $\wh{\theta}_{n}$ or $\wt{\theta}_{n}$, and $(\bar{\tau}_{2,n}, \bar{\sigma}_{2,n})$ is the couple of estimates suggested by \cite{NichollsQuinn81b} in formulas (3.6) and (3.7) respectively, also given in \cite{Jurgens85}. They are defined as
\begin{equation}
\label{EstVarH0}
\bar{\tau}_{2,n} = \frac{\sum_{t=1}^{n} (Z_{t} - \bar{Z}_{n})\, \wh{\veps}_{t}^{~ 2}}{\sum_{t=1}^{n} (Z_{t} - \bar{Z}_{n})^2} \hsp \text{and} \hsp \bar{\sigma}_{2,n} = \wh{\sigma}_{2, n} - \bar{Z}_{n}\, \bar{\tau}_{2,n}
\end{equation}
where $(\wh{\veps}_{t})$ is the residual set built in \eqref{ResSet}, $\wh{\sigma}_{2, n}$ is given in \eqref{EstOLSVar} and for $t \in \{1, \hdots, n\}$, $Z_{t} = X_{t}^2$. Thm. 4.2 of \cite{NichollsQuinn81b} gives their consistency as soon as the RCAR(1) process has fourth-order moments. Furthermore, our study gives the consistency of $\bar{\theta}_{n}$ under $\cH_0$. We deduce from Slutsky's lemma that
\begin{equation}
\label{StatTest}
\wh{\psi}_{n}^{\, 0} \cvgas \psi^{\, 0} > 0 \hsp \text{and} \hsp 
\frac{n\, \big( \wt{\gamma}_{n} \big)^2}{\wh{\psi}_{n}^{\, 0}} \cvgl \chi^2_1
\end{equation}
if $\cH_0$ is true, where $\chi^2_1$ has a chi-square distribution with one degree of freedom, whereas under $\cH_1$ the test statistic diverges (almost surely). The introduction of (H$_5$) enables to choose
\begin{equation*}
\bar{\sigma}_{4,n} = g(\bar{\sigma}_{2,n}) \hsp \text{and} \hsp \bar{\tau}_{4,n} = h(\bar{\tau}_{2,n})
\end{equation*}
as consistent estimations of the related moments. Comparing the test statistic with the quantiles of $\chi^2_1$ may constitute the basis of a test for the existence of correlation in the random coefficients of an autoregressive process. To conclude, we have shown through this simple model that the introduction of correlation in the coefficients is a significative issue in relation to the inference procedure. And yet, in a time series context it seems quite natural to take account of autocorrelation in the random coefficients, this is an incitement to put statistical conclusions into perspective dealing with estimation and testing procedures of RCAR models. The most challenging extensions for future studies seem to rely on more complex dependency structures in the coefficients, on the consideration of more autoregressions in the model, and of course on the behavior of the process under instability and unit root issues. The testing procedure for correlation in the random coefficients should also be studied on an empirical basis, this is an ongoing investigation.

\medskip

\noindent \textbf{Acknowledgments.} The authors thank the Associate Editor and the anonymous Reviewer for the suggestions and very constructive comments which helped to improve substantially the paper.
%%%%%%%%%%%%%%%%%%%%%%%%%%%%%%%%%%%%%%%%%%%%%%%%%%%%%%%%%%%%%%%%%%%%%%%%%%%%%%%%

\section{Proofs of the main results}
\label{SecProofs}

In this section, we develop the whole proofs of our results. The fundamental tools related to ergodicity may be found in Thm. 3.5.8 of \cite{Stout74} or in Thm. 1.3.3 of \cite{TaniguchiKakizawa00}. We will repeatedly have to deal with $\dE[\eta^{\, a}_{t} (\theta + \eta_{t})^{b}]$ for $a,b \in \{0, \hdots, 4\}$, so we found useful to summarize beforehand the associated values under (H$_2$) in Table \ref{TabExp} below.
\begin{table}[h!]
\begin{center}
\begin{tabular}{|c|c|c|c|c|c|}
\hline
$a \backslash b$ & 0 & 1 & 2 & 3 & 4 \\
\hline
0 & 1 & $\theta$ & $\theta^2 + \tau_2$ & $\theta^3 + 3\, \theta\, \tau_2$ & $\theta^4 + 6\, \theta^2\, \tau_2 + \tau_4$ \\
\hline
1 & 0 & $\tau_2$ & $2\, \theta\, \tau_2$ & $3\, \theta^2\, \tau_2 + \tau_4$ & $4\, \theta^3\, \tau_2 + 4\, \theta\, \tau_4$ \\
\hline
2 & $\tau_2$ & $\theta\, \tau_2$ & $\theta^2\, \tau_2 + \tau_4$ & $\theta^3\, \tau_2 + 3\, \theta\, \tau_4$ & $\theta^4\, \tau_2 + 6\, \theta^2\, \tau_4 + \tau_6$ \\
\hline
3 & 0 & $\tau_4$ & $2\, \theta\, \tau_4$ & $3\, \theta^2\, \tau_4 + \tau_6$ & $4\, \theta^3\, \tau_4 + 4\, \theta\, \tau_6$ \\
\hline
4 & $\tau_4$ & $\theta\, \tau_4$ & $\theta^2\, \tau_4 + \tau_6$ & $\theta^3\, \tau_4 + 3\, \theta\, \tau_6$ & $\theta^4\, \tau_4 + 6\, \theta^2\, \tau_6 + \tau_8$ \\
\hline
\end{tabular}
\bigskip
\end{center}
\caption{\small $\dE[\eta^{\, a}_{t} (\theta + \eta_{t})^{b}]$ for $a,b \in \{0, \hdots, 4\}$.}
\label{TabExp}
\end{table}
For the sake of clarity, we postpone to the appendix the numerous constants that will be used thereafter. We start by giving some short explanations related to the remarks appearing in Sections \ref{SecProp} and \ref{SecOLS}.

\subsection{About the remarks of Sections \ref{SecProp} and \ref{SecOLS}}
\label{SecProofRem}

\subsubsection{Remark \ref{RemPartCase}} Indeed, the explicit calculation of $\gamma_{X}(0)$ based on \eqref{ACV} leads to
\begin{equation*}
\gamma_{X}(0) = \frac{\sigma_2\, (2\, \alpha\, \tau_2 - 1)}{d(\theta, \alpha, \tau_2, \tau_4)}
\end{equation*}
for some denominator satisfying $d(\theta, \alpha, \tau_2, \tau_4) = 2\, \theta^2$ when $2\, \alpha\, \tau_2 = 1$. It follows that should this assumption be true under second-order stationarity, the process would be deterministic.

\subsubsection{Remark \ref{RemInitCond}} The objective here is to show that the difference between the process starting at $X_0$ having the strictly stationary and ergodic distribution and the same process starting at some $Y_0$ is (a.s.) negligible provided very weak assumptions on $Y_0$. Following the idea of Lem. 1 in \cite{AueHorvathSteinebach06} and using the ergodic theorem, we obtain that for a sufficiently large $t$, almost surely
\begin{equation*}
\frac{1}{t} \sum_{\ell=1}^{t} \ln \vert \theta + \alpha\, \eta_{\ell-1} + \eta_{\ell} \vert \leq \frac{\kappa}{2} < 0.
\end{equation*}
Hence, the asymptotic decrease of $\prod_{\ell=1}^{t} \vert \theta + \alpha\, \eta_{\ell-1} + \eta_{\ell} \vert$ is exponentially fast with $t$ under (H$_1$) and the upper bound of $\vert X_{t}-Y_{t} \vert \leq \vert X_0-Y_0 \vert\, \de^{\frac{\kappa\, t}{2}}$ enables to retain weak assumptions on $Y_0$ so that $X_{t}-Y_{t} = o(1)$ a.s.

\subsubsection{Remark \ref{RemPartCaseEmpMean}} In the particular case where $\alpha = \tau_2 = 0$ (that is, in the stable AR(1) process), Thm. 7.1.2 of \cite{BrockwellDavis96} states that $\sqrt{n}\, \bar{X}_{n}$ is asymptotically normal with mean 0 and variance given by
\begin{equation*}
\sum_{h\, \in\, \dZ} \gamma_{X}(h) = \sigma_2\, \Big( \sum_{k=0}^{+\infty} \theta^{k} \Big)^{\! 2} = \frac{\sigma_2}{(1-\theta)^2}.
\end{equation*}
Thus, $\kappa^2_{00}$ implied by our results is coherent from that point of view.

\subsubsection{Remark \ref{RemPartCaseOLS}} Like in the previous remark, Prop. 8.10.1 of \cite{BrockwellDavis96} states that, for $\alpha = \tau_2 = 0$, the OLS estimator of $\theta$ is asymptotically normal with rate $\sqrt{n}$, mean 0 and variance given by $1-\theta^2$, which corresponds to $\omega^2_{00}$. Now if $\tau_2 > 0$, Thm. 4.1 of \cite{NichollsQuinn81b}, and especially formula (4.1), gives the asymptotic variance as a function of $\dE[X_{t}^2]$ and $\dE[X_{t}^4]$ as detailed in Rem. \ref{RemPartCaseOLS}. Our study enables to identify $\omega^2_{0}$ as a function of the parameters by injecting $\alpha=0$ into $\lambda_0$ and $\delta_0$ that are computed in \eqref{LamExpl} and \eqref{DelExpl}, respectively.

\subsection{Proof of Theorem \ref{ThmCausal}}
\label{SecProofThmCausal}
The existence of the almost sure causal representation of $(X_{t})$ under (H$_1$) is a corollary of Thm. 1 of \cite{Brandt86}. Indeed, $(\theta_{t})$ is a stationary and ergodic MA(1) process independent of $(\veps_{t})$, itself obviously stationary and ergodic. Let us give more details. First, hypotheses (H$_1$) enable to make use of the same proof as \cite{AueHorvathSteinebach06} where the ergodic theorem replaces the strong law of large numbers to reach formula (6), and to establish that \eqref{XCausal} is the limit of a convergent series (with probability 1). Then for all $t \in \dZ$,
\begin{eqnarray*}
\theta_{t}\, X_{t-1} & = & (\theta + \alpha\, \eta_{t-1} + \eta_{t}) \left[ \veps_{t-1} + \sum_{k=1}^{\infty} \veps_{t-k-1}\, \prod_{\ell=0}^{k-1} (\theta + \alpha\, \eta_{t-\ell-2} + \eta_{t-\ell-1}) \right] \\
 & = & \sum_{k=1}^{\infty} \veps_{t-k}\, \prod_{\ell=0}^{k-1} (\theta + \alpha\, \eta_{t-\ell-1} + \eta_{t-\ell}) ~ = ~ X_{t} - \veps_{t}
\end{eqnarray*}
meaning that \eqref{XCausal} is a solution to the recurrence equation. Finally, the strict stationarity and ergodicity of $(X_{t})$ may be obtained following the same reasoning as in \cite{NichollsQuinn81b}. Indeed, the causal representation \eqref{XCausal} shows that there exists $\phi$ independent of $t$ such that for all $t \in \dZ$,
$$
X_{t} = \phi((\veps_{t}, \eta_{t}), (\veps_{t-1}, \eta_{t-1}), \hdots).
$$
The set $((\veps_{t}, \eta_{t}), (\veps_{t-1}, \eta_{t-1}), \hdots)$ being made of independent and identically distributed random vectors, $(X_{t})$ is strictly stationary. The ergodicity follows from Thm. 1.3.3 of \cite{TaniguchiKakizawa00}. $\hfill\qed$

\subsection{Proof of Theorem \ref{ThmStat}}
\label{SecProofThmStat}
Ergodicity and strict stationarity come from Theorem \ref{ThmCausal}. We consider the causal representation \eqref{XCausal}. First, since $(\veps_{t})$ and $(\eta_{t})$ are uncorrelated white noises, for all $t \in \dZ$,
\begin{equation}
\label{StatEsp}
\dE[X_{t}] = 0.
\end{equation}
To establish the autocovariance function of $(X_{t})$, we have beforehand to establish a technical lemma related to the second-order properties of the process. For all $k, h \in \dN^{*}$, consider the sequence
\begin{eqnarray*}
u_{0,h}^{(a)} & = & \dE[\eta_{h}^{a}\, \theta_{h}\, \hdots\, \theta_{1}], \\
u_{k,0}^{(a)} & = & \dE[\eta_{k}^{a}\, \theta_{k}^2\, \hdots\, \theta_{1}^2], \\
u_{k,h}^{(a)} & = & \dE[\eta_{k+h}^{a}\, \theta_{k+h}\, \hdots\, \theta_{k+1}\, \theta_{k}^2\, \hdots\, \theta_{1}^2],
\end{eqnarray*}
where $a \in \{ 0,1,2 \}$, and build
\begin{equation}
\label{Ukh}
U_{k,h} = \begin{pmatrix}
u_{k,h}^{(0)} \\
u_{k,h}^{(1)} \\
u_{k,h}^{(2)}
\end{pmatrix}.
\end{equation}
Thereafter, $M$, $N$ and $U_0$ refer to \eqref{M}, \eqref{N} and \eqref{FamVecMom2}, respectively.
\begin{lem}
\label{LemSeq2}
Assume that (H$_1$)--(H$_3$) hold. Then, for all $h, k \in \dN$,
\begin{equation}
\label{SystU2}
U_{k,h} = N^{h}\, M^{k}\, U_0
\end{equation}
with the convention that $U_{0,0} = U_0$.
\end{lem}
\begin{proof}
In the whole proof, $(\cF_{t})$ is the filtration defined in \eqref{Filtr} and Table \ref{TabExp} may be read to compute the coefficients appearing in the calculations. The coefficients $\theta_{k+h-1}$, $\theta_{k+h-2}$, $\hdots$ are $\cF_{k+h-1}$--measurable. Hence for $h \geq 1$,
\begin{eqnarray*}
u_{k,h}^{(0)} & = & \dE[\theta_{k+h-1}\, \hdots\, \theta_{k+1}\, \theta_{k}^2\, \hdots\, \theta_{1}^2\, \dE[\theta_{k+h}\, \vert\, \cF_{k+h-1}]] \\
 & = & \theta\, u_{k,h-1}^{(0)} + \alpha\, u_{k,h-1}^{(1)}, \\
u_{k,h}^{(1)} & = & \dE[\theta_{k+h-1}\, \hdots\, \theta_{k+1}\, \theta_{k}^2\, \hdots\, \theta_{1}^2\, \dE[\eta_{k+h}\, \theta_{k+h}\, \vert\, \cF_{k+h-1}]] \\
 & = & \tau_2\, u_{k,h-1}^{(0)}, \\
u_{k,h}^{(2)} & = & \dE[\theta_{k+h-1}\, \hdots\, \theta_{k+1}\, \theta_{k}^2\, \hdots\, \theta_{1}^2\, \dE[\eta_{k+h}^2\, \theta_{k+h}\, \vert\, \cF_{k+h-1}]] \\
 & = & \theta\, \tau_2\, u_{k,h-1}^{(0)} + \alpha\, \tau_2\, u_{k,h-1}^{(1)}.
\end{eqnarray*}
We get the matrix formulation $U_{k,h} = N\, U_{k,h-1}$. It follows that, for $h \in \dN$,
\begin{equation}
\label{SystStatU1}
U_{k,h} = N^{h}\, U_{k,0}.
\end{equation}
The next step is to compute $U_{k,0}$, and we will use the same lines. For $k \geq 1$,
\begin{eqnarray*}
u_{k,0}^{(0)} & = & \dE[\theta_{k-1}^2\, \hdots\, \theta_{1}^2\, \dE[\theta_{k}^2\, \vert\, \cF_{k-1}]] \\
 & = & (\theta^2 + \tau_2)\, u_{k-1,0}^{(0)} + 2\, \alpha\, \theta\, u_{k-1,0}^{(1)} + \alpha^2\, u_{k-1,0}^{(2)}, \\
u_{k,0}^{(1)} & = & \dE[\theta_{k-1}^2\, \hdots\, \theta_{1}^2\, \dE[\eta_{k}\, \theta_{k}^2\, \vert\, \cF_{k-1}]] \\
 & = & 2\, \theta\, \tau_2\, u_{k-1,0}^{(0)} + 2\, \alpha\, \tau_2\, u_{k-1,0}^{(1)}, \\
u_{k,0}^{(2)} & = & \dE[\theta_{k-1}^2\, \hdots\, \theta_{1}^2\, \dE[\eta_{k}^2\, \theta_{k}^2\, \vert\, \cF_{k-1}]] \\
 & = & (\theta^2\, \tau_2 + \tau_4)\, u_{k-1,0}^{(0)} + 2\, \alpha\, \theta\, \tau_2\, u_{k-1,0}^{(1)} + \alpha^2\, \tau_2\, u_{k-1,0}^{(2)}.
\end{eqnarray*}
Thus, \eqref{SystStatU1} becomes
$$
U_{k,h} = N^{h}\, M^{k-1}\, U_{1,0}
$$
where the initial vector $U_{1,0}$ is given by
\begin{eqnarray*}
u_{1,0}^{(0)} & = & \dE[\theta_{1}^2] ~ = ~ (\theta^2 + \tau_2) + \alpha^2\, \tau_2, \\
u_{1,0}^{(1)} & = & \dE[\eta_{1}\, \theta_{1}^2] ~ = ~ 2\, \theta\, \tau_2, \\
u_{1,0}^{(2)} & = & \dE[\eta_{1}^2\, \theta_{1}^2] ~ = ~ (\theta^2\, \tau_2 + \tau_4) + \alpha^2\, \tau_2^2.
\end{eqnarray*}
It is then not hard to conclude that, for all $k \in \dN^{*}$ and $h \in \dN$,
$$
U_{k,h} = N^{h}\, M^{k}\, U_0.
$$
For $k=0$, a similar calculation based on the initial values $u_{0,h}^{(a)}$ for $a \in \{ 0,1,2 \}$ leads to $U_{0,h} = N^{h}\, U_0$, implying that \eqref{SystU2} holds for all $k, h \in \dN$.
\end{proof}

\begin{cor}
\label{CorMom2}
Assume that (H$_1$)--(H$_3$) hold. Then, the second-order properties of $(X_{t})$ are such that, for all $a \in \{ 0,1,2 \}$,
$$
\dE[\eta_{t}^{a}\, X_{t}^2] < \infty. 
$$
\end{cor}
\begin{proof}
For all $t \in \dZ$ and $k \geq 1$, denote by
\begin{equation}
\label{Mom2}
\Lambda_{t} = \begin{pmatrix}
1 \\ \eta_{t} \\ \eta_{t}^2
\end{pmatrix} \hspand P_{t,\,k} = \prod_{i=0}^{k-1} \theta_{t-i}
\end{equation}
with $P_{t,\,0} = 1$. Since $(\veps_{t})$ and $(\eta_{t})$ are uncorrelated white noises, using the causal representation \eqref{XCausal} and letting $h=0$,
$$
\dE[\Lambda_{t}\, X_{t}^2] = \sum_{k=0}^{\infty} \sum_{\ell=0}^{\infty} \dE[\Lambda_{t}\, P_{t,\,k}\, P_{t,\,\ell}\, \veps_{t-k}\, \veps_{t-\ell}] = \sigma_2\, \sum_{k=0}^{\infty} M^{k}\, U_0 = \sigma_2\, (I_3-M)^{-1}\, U_0
$$
as a consequence of the strict stationarity of $(\theta_{t})$. We remind that, under (H$_3$), it is well-known (see \textit{e.g.} \cite{HornJohnson85}) that $I_3-M$ is invertible and that
$$
\sum_{k=0}^{\infty} M^{k} = (I_3-M)^{-1}.
$$
\end{proof}

Let us return to the proof of Theorem \ref{ThmStat}. From Lemma \ref{LemSeq2} and Corollary \ref{CorMom2}, we are now able to evaluate the autocovariance function of $(X_{t})$. For $h \in \dN$,
$$
\cov(X_{t}, X_{t-h}) = \sum_{k=0}^{\infty} \sum_{\ell=0}^{\infty} \dE[P_{t,\,k}\, P_{t-h,\,\ell}\, \veps_{t-k}\, \veps_{t-h-\ell}].
$$
We get
$$
\gamma_{X}(h) = \sigma_2\, \sum_{k=0}^{\infty} \dE[P_{t,\,k+h}\, P_{t-h,\,k}] = \sigma_2\, \Big( \dE[P_{t,\,h}] + \sum_{k=1}^{\infty} u_{k,h}^{(0)} \Big) = \sigma_2\, \Big[ \sum_{k=0}^{\infty} U_{k,h} \Big]_{1}.
$$
From Lemma \ref{LemSeq2},
$$
\gamma_{X}(h) = \sigma_2\, \big[ N^{h}\, (I_3-M)^{-1}\, U_0 \big]_1.
$$
We conclude using the fact that $\gamma_{X}$ does not depend on $t$. For all $t \in \dZ$ and $h \in \dN$, $\gamma_{X}(h) = \cov(X_{t-h}, X_{t}) = \cov(X_{t}, X_{t+h})$, which shows that the above reasoning still holds for $h \in \dZ$, replacing $h$ by $\vert h \vert$. Now suppose that $(W_{t})$ is another causal ergodic strictly and second-order stationary solution. There exists $\varphi$ independent of $t$ such that for all $t \in \dZ$,
$$
X_{t} - W_{t} = \varphi((\veps_{t}, \eta_{t}), (\veps_{t-1}, \eta_{t-1}), \hdots)
$$
and necessarily, $(X_{t}-W_{t})$ is also a strictly stationary process having second-order moments. Let $e^{(a)} = \dE[\eta_{t}^{a}\, (X_{t}-W_{t})^2]$, for $a \in \{ 0,1,2 \}$. From the same calculations and exploiting the second-order stationarity of $(X_{t}-W_{t})$, it follows that
\begin{equation*}
\begin{pmatrix}
e^{(0)} \\
e^{(1)} \\
e^{(2)}
\end{pmatrix} = M \begin{pmatrix}
e^{(0)} \\
e^{(1)} \\
e^{(2)}
\end{pmatrix}
\end{equation*}
implying, if $(e^{(0)} ~ e^{(1)} ~ e^{(2)}) \neq 0$, that $1$ is an eigenvalue of $M$. Clearly, this contradicts $\rho(M) < 1$ which is part of (H$_3$). Thus, $\dE[(X_{t}-W_{t})^2]$ must be zero and $X_{t} = W_{t}$ a.s.

$\hfill\qed$

\subsection{Proof of Theorem \ref{ThmEmpMean}}
\label{SecProofThmEmpMean}
The convergence to zero is only the application of the ergodic theorem, since we have seen in \eqref{StatEsp} that $\dE[X_{t}] = 0$. Here, only (H$_1$) and (H$_2$) are needed. We make the following notations,
\begin{eqnarray*}
\bar{M}_{n}^{(1)} & = & \sum_{t=1}^{n} X_{t-1}\, \big( (1 + \alpha\, \theta)\, \eta_{t} + \alpha\, (\eta_{t}^2 - \tau_2) \big), \\
\bar{M}_{n}^{(2)} & = & \alpha^2 \sum_{t=1}^{n} \eta_{t-1}\, X_{t-1}\, \eta_{t}, \\
\bar{M}_{n}^{(3)} & = & \sum_{t=1}^{n} (1 + \alpha\, \eta_{t}) \, \veps_{t}.
\end{eqnarray*}
Consider the filtration $(\cF^{\, *}_{n})$ generated by $\cF^{\, *}_0 = \sigma(X_0, \eta_0)$ and, for $n \geq 1$, by
\begin{equation}
\label{FiltrInf}
\cF^{\, *}_{n} = \sigma( X_0, \eta_0, (\veps_1, \eta_1), \hdots, (\veps_{n}, \eta_{n}) )
\end{equation}
and let
\begin{equation}
\label{VecMartEmpMean}
\bar{\cM}_{n} = \begin{pmatrix}
\bar{M}_{n}^{(1)} \\
\bar{M}_{n}^{(2)} \\
\bar{M}_{n}^{(3)}
\end{pmatrix}.
\end{equation}
Under our hypotheses, $\bar{\cM}_{n}$ is a locally square-integrable real vector $(\cF^{\, *}_{n})$--martingale. We shall make use of the central limit theorem for vector martingales given \textit{e.g.} by Cor. 2.1.10 of \cite{Duflo97}. On the one hand, we have to study the asymptotic behavior of the predictable
quadratic variation of $\bar{\cM}_{n}$. For all $n \geq 1$, let
\begin{equation}
\label{MartQuadEmpMean}
\langle \bar{\cM} \rangle_{n} = \sum_{t=1}^{n} \dE\big[ (\Delta \bar{\cM}_{t} )(\Delta \bar{\cM}_{t} )^{T}\, \vert\, \cF^{\, *}_{t-1} \big],
\end{equation}
with $\Delta \bar{\cM}_1 = \bar{\cM}_1$. To simplify the calculations, we introduce some more notations. The second-order moments of the process are called
\begin{equation}
\label{Lam}
\dE[\Lambda_{t}\, X_{t}^2] = \begin{pmatrix}
\lambda_{0} \\
\lambda_{1} \\
\lambda_{2}
\end{pmatrix} = \Lambda
\end{equation}
where $\Lambda_{t}$ is given in \eqref{Mom2}, with $\lambda_0 = \gamma_{X}(0)$. We use the strict stationarity to establish, following Corollary \ref{CorMom2} under the additional (H$_3$) hypothesis, that
\begin{equation}
\label{LamExpl}
\Lambda = \sigma_2\, (I_3-M)^{-1}\, U_0
\end{equation}
and ergodicity immediately leads to
\begin{equation}
\label{CvgEmpMom2}
\frac{1}{n} \sum_{t=1}^{n} \Lambda_{t}\, X_{t}^2 \cvgas \Lambda.
\end{equation}
Now, we are going to study the asymptotic behavior of $\langle \bar{\cM} \rangle_{n}/n$. First, under our assumptions,
$$
\langle \bar{M}^{(1)}, \bar{M}^{(3)} \rangle_{n} = \langle \bar{M}^{(2)}, \bar{M}^{(3)} \rangle_{n} = 0.
$$
Since the other calculations are very similar we only detail the first one,
\begin{eqnarray*}
\langle \bar{M}^{(1)} \rangle_{n} & = & \sum_{t=1}^{n} X_{t-1}^{\, 2}\, \dE\big[ \big( (1 + \alpha\, \theta)\, \eta_{t} + \alpha\, (\eta_{t}^2 - \tau_2) \big)^2 \big] \\
& = & \big( (1 + \alpha\, \theta)^2\, \tau_2 + \alpha^2\, (\tau_4 - \tau_2^2) \big) \sum_{t=1}^{n} X_{t-1}^{\, 2}.
\end{eqnarray*}
We obtain using $\bar{K}$ in \eqref{Kb} that
\begin{equation}
\label{MartQuadExplEmpMean}
\langle \bar{\cM} \rangle_{n} = \bar{K} \circ \sum_{t=1}^{n} \begin{pmatrix}
X_{t}^2 & \eta_{t}\, X_{t}^2 & 0 \\
\eta_{t}\, X_{t}^2 & \eta_{t}^2\, X_{t}^2 & 0\\
0 & 0 & 1
\end{pmatrix} + \bar{R}_{n}
\end{equation}
where the Hadamard product $\circ$ is used to lighten the formula, and where the remainder $\bar{R}_{n}$ is made of isolated terms such that, from \eqref{CvgEmpMom2},
\begin{equation}
\label{MartQuadCvgRemEmpMean}
\frac{\bar{R}_{n}}{n} \cvgas 0.
\end{equation}
We reach these results by computing $\langle \bar{M}^{(i)}, \bar{M}^{(j)} \rangle_{n}$ for $i,j \in \{ 1, 2, 3 \}$ just as we have done above for some of them, and then by normalizing each sum, leaving the isolated terms in the remainder. For example,
\begin{equation*}
\sum_{t=1}^{n} X_{t-1}^{\, 2} = \sum_{t=1}^{n} X_{t}^{\, 2} + (X_0^2 - X_{n}^{\, 2}).
\end{equation*}
It is then a direct application of the ergodic theorem that gives the $O(n)$ behavior of the sums (and the $o(n)$ behavior of the isolated terms as a consequence), and that enables to identify, by combining \eqref{CvgEmpMom2}, \eqref{MartQuadExplEmpMean} and \eqref{MartQuadCvgRemEmpMean}, the limiting value
\begin{equation}
\label{MartQuadCvgEmpMean}
\frac{\langle \bar{\cM} \rangle_{n}}{n} \cvgas \bar{K} \circ \bar{\Gamma}
\end{equation}
where $\bar{\Gamma}$ is given by
\begin{equation}
\label{Gamb}
\bar{\Gamma} = \begin{pmatrix}
\lambda_0 & \lambda_1 & 0 \\
\lambda_1 & \lambda_2 & 0 \\
0 & 0 & 1
\end{pmatrix}.
\end{equation}
On the other hand, it is necessary to prove that the Lindeberg's condition is satisfied, namely that for all $\veps > 0$,
\begin{equation}
\label{MartLindEmpMean}
\frac{1}{n} \sum_{t=1}^{n} \dE\big[ \Vert \Delta \bar{\cM}_{t} \Vert^2\, \dI_{ \{ \Vert \Delta \bar{\cM}_{t} \Vert\, \geq\, \veps\, \sqrt{n} \}}\, \vert\, \cF^{\, *}_{t-1} \big] \cvgp 0
\end{equation}
as $n$ tends to infinity. By ergodicity and strict stationarity of the increments $(\Delta \bar{\cM}_{t})$ under the assumption on $X_0$, it follows that for any $M > 0$,
$$
\frac{1}{n} \sum_{t=1}^{n} \dE\big[ \Vert \Delta \bar{\cM}_{t} \Vert^2\, \dI_{ \{ \Vert \Delta \bar{\cM}_{t} \Vert\, \geq\, M \}}\, \vert\, \cF_{t-1}^{\, *} \big] \cvgas \dE\big[ \Vert \Delta \bar{\cM}_1 \Vert^2\, \dI_{ \{ \Vert \Delta \bar{\cM}_1 \Vert\, \geq\, M \}} \big].
$$
Corollary \ref{CorMom2} implies that $\dE[\Vert \Delta \bar{\cM}_1 \Vert^2] < \infty$ and the right-hand side can be made arbitrarily small, which establishes the Lindeberg's condition. From \eqref{MartQuadCvgEmpMean} and \eqref{MartLindEmpMean}, we deduce that
\begin{equation}
\label{MartCvgEmpMean}
\frac{\bar{\cM}_{n}}{\sqrt{n}} \cvgl \cN(0, \bar{K} \circ \bar{\Gamma})
\end{equation}
which is nothing but the central limit theorem for vector martingales, as intended. One can notice that the above reasoning is in fact a vector extension of the main result of \cite{Billingsley61}, related to the central limit theorem for martingales having ergodic and stationary increments. Finally, by a tedious but straightforward calculation, one can obtain that
$$
\sqrt{n}\, \bar{X}_{n} = \frac{\Omega_{3}^{\, T} \bar{\cM}_{n} + \bar{r}_{n}}{(1 - \theta - \alpha\, \tau_2)\, \sqrt{n}}
$$
where $\Omega_{3}^{\, T} = (1 ~ 1 ~ 1)$ and $\bar{r}_{n} = o(\sqrt{n})$ a.s. from \eqref{CvgEmpMom2}. It remains to apply Slutsky's lemma to conclude that
$$
\sqrt{n}\, \bar{X}_{n} \cvgl \cN(0, \kappa^2)
$$
with
\begin{equation}
\label{VarEmpMean}
\kappa^2 = \frac{\Omega_{3}^{\, T} (\bar{K} \circ \bar{\Gamma})\, \Omega_{3}}{(1 - \theta - \alpha\, \tau_2)^2}
\end{equation}
using the whole notations above.
$\hfill\qed$

\subsection{Proof of Theorem \ref{ThmOLS}}
\label{SecProofThmOLS}
The almost sure convergence essentially relies on the ergodicity of the process. Theorem \ref{ThmStat} together with the ergodic theorem directly lead to
$$
\wh{\theta}_{n} \cvgas \frac{\gamma_{X}(1)}{\gamma_{X}(0)} = \frac{\big[ N\, (I_3-M)^{-1}\, U_0 \big]_1}{\big[ (I_3-M)^{-1}\, U_0 \big]_1}
$$
as $n$ tends to infinity, but we are interested in the explicit form of the limiting value. From the combined expressions \eqref{XAR}--\eqref{MACoef},  it follows that
\begin{equation}
\label{CvgDecomp}
\sum_{t=1}^{n} X_{t-1} X_{t} = \theta\, \sum_{t=1}^{n} X_{t-1}^{\, 2} + \alpha\, \sum_{t=1}^{n} \eta_{t-1}\, X_{t-1}^{\, 2} + \sum_{t=1}^{n} X_{t-1}^{\, 2}\, \eta_{t} + \sum_{t=1}^{n} X_{t-1}\, \veps_{t}.
\end{equation}
We also note from Corollary \ref{CorMom2} that, for all $t \in \dZ$,
\begin{eqnarray*}
\dE[\eta_{t}\, X_{t}^2] & = & \dE[\theta_{t}^2\, X_{t-1}^{\, 2}\, \eta_{t}] + \dE[\veps_{t}^2\, \eta_{t}] + 2\, \dE[\theta_{t}\, X_{t-1}\, \veps_{t}\, \eta_{t}] \\
 & = & 2\, \alpha\, \tau_2\, \dE[\eta_{t-1}\, X_{t-1}^{\, 2}] + 2\, \theta\, \tau_2\, \dE[X_{t-1}^{\, 2}].
\end{eqnarray*}
Thus, by stationarity and ergodicity,
\begin{equation}
\label{CvgT1}
\frac{1}{n} \sum_{t=1}^{n} \eta_{t-1}\, X_{t-1}^{\, 2} \cvgas \frac{2\, \theta\, \tau_2\, \gamma_{X}(0)}{1 - 2\, \alpha\, \tau_2}.
\end{equation}
Similarly, $\dE[X_{t-1}^{\, 2}\, \eta_{t}] = \dE[X_{t-1}\, \veps_{t}] = 0$ and from the ergodic theorem,
\begin{equation}
\label{CvgT2}
\frac{1}{n} \sum_{t=1}^{n} X_{t-1}^{\, 2} \cvgas \gamma_{X}(0), \hsp \frac{1}{n} \sum_{t=1}^{n} X_{t-1}^{\, 2}\, \eta_{t} \cvgas 0, \hsp \frac{1}{n} \sum_{t=1}^{n} X_{t-1}\, \veps_{t} \cvgas 0.
\end{equation}
The expression of $\wh{\theta}_{n}$ in \eqref{EstOLS} combined with the decomposition \eqref{CvgDecomp} and the convergences \eqref{CvgT1} and \eqref{CvgT2} give
$$
\wh{\theta}_{n} \cvgas \theta + \frac{2\, \alpha\, \theta\, \tau_2}{1 - 2\, \alpha\, \tau_2} = \frac{\theta}{1 - 2\, \alpha\, \tau_2}.
$$
Let us now establish the asymptotic normality. First, we have to study the fourth-order properties of $(X_{t})$ and some other technical lemmas are needed. For all $k \in \dN^{*}$, consider the sequences
\begin{eqnarray*}
v_{k}^{(a)} & = & \dE[\eta_{k}^{a}\, \theta_{k}^4\, \hdots\, \theta_{1}^4]
\end{eqnarray*}
where $a \in \{ 0, \hdots, 4 \}$, and build
\begin{equation}
\label{V}
V_{k} = \begin{pmatrix}
v_{k}^{(0)} \\
\vdots \\
v_{k}^{(4)}
\end{pmatrix}.
\end{equation}
For the following calculations, $H$ is defined in \eqref{H} and $\{ V_0, \hdots, V_4 \}$ in \eqref{FamVecMom4}.
\begin{lem}
\label{LemSeq4}
Assume that (H$_1$)--(H$_4$) hold. Then, for all $k \in \dN$,
\begin{equation}
\label{SystV4}
V_{k} = H^{k}\, V_0.
\end{equation}
\end{lem}
\begin{proof}
With the filtration $(\cF_{t})$ defined in \eqref{Filtr}, for $k \geq 1$,
\begin{eqnarray*}
v_{k}^{(0)} & = & \dE[\theta_{k-1}^4\, \hdots\, \theta_{1}^4\, \dE[\theta_{k}^4\, \vert\, \cF_{k-1}]] \\
 & = & (\theta^4 + 6\, \theta^2\, \tau_2 + \tau_4)\, v_{k-1}^{(0)} + 4\, \alpha\,(\theta^3 + 3\, \theta \,\tau_2)\, v_{k-1}^{(1)} + 6\, \alpha^2\,(\theta^2 + \tau_2)\, v_{k-1}^{(2)} \\
 & & \hsp \hsp \hsp + ~ 4\, \alpha^3\, \theta\, v_{k-1}^{(3)} + \alpha^4\, v_{k-1}^{(4)}, \\
v_{k}^{(1)} & = & \dE[\theta_{k-1}^4\, \hdots\, \theta_{1}^4\, \dE[\eta_{k}\, \theta_{k}^4\, \vert\, \cF_{k-1}]] \\
 & = & (4\, \theta^3\, \tau_2 + 4\, \theta\, \tau_4)\, v_{k-1}^{(0)} + 4\, \alpha\,(3\, \theta^2\, \tau_2 + \tau_4)\, v_{k-1}^{(1)} + 12\, \alpha^2\, \theta\, \tau_2\, v_{k-1}^{(2)} \\
 & & \hsp \hsp \hsp + ~ 4\, \alpha^3\, \tau_2\, v_{k-1}^{(3)}, \\
v_{k}^{(2)} & = & \dE[\theta_{k-1}^4\, \hdots\, \theta_{1}^4\, \dE[\eta_{k}^2\, \theta_{k}^4\, \vert\, \cF_{k-1}]] \\
 & = & (\theta^4\, \tau_2 + 6\, \theta^2\, \tau_4 + \tau_6)\, v_{k-1}^{(0)} + 4\, \alpha\,(\theta^3\, \tau_2 + 3\, \theta\, \tau_4) \, v_{k-1}^{(1)} + 6\, \alpha^2\, (\theta^2\, \tau_2 + \tau_4)\, v_{k-1}^{(2)} \\
 & & \hsp \hsp \hsp + ~ 4\, \alpha^3\, \theta\, \tau_2\, v_{k-1}^{(3)} + \alpha^4\, \tau_2\, v_{k-1}^{(4)}, \\
v_{k}^{(3)} & = & \dE[\theta_{k-1}^4\, \hdots\, \theta_{1}^4\, \dE[\eta_{k}^3\, \theta_{k}^4\, \vert\, \cF_{k-1}]] \\
 & = & (4\, \theta^3\, \tau_4 + 4\, \theta\, \tau_6)\, v_{k-1}^{(0)} + 4\, \alpha\,(3\, \theta^2\, \tau_4 + \tau_6)\, v_{k-1}^{(1)} + 12\, \alpha^2\, \theta\, \tau_4\, v_{k-1}^{(2)} \\
 & & \hsp \hsp \hsp + ~ 4\, \alpha^3\, \tau_4\, v_{k-1}^{(3)}, \\
v_{k}^{(4)} & = & \dE[\theta_{k-1}^4\, \hdots\, \theta_{1}^4\, \dE[\eta_{k}^4\, \theta_{k}^4\, \vert\, \cF_{k-1}]] \\
 & = & (\theta^4\, \tau_4 + 6\, \theta^2\, \tau_6 + \tau_8)\, v_{k-1}^{(0)} + 4\, \alpha\,(\theta^3\, \tau_4 + 3\, \theta\, \tau_6) \, v_{k-1}^{(1)} + 6\, \alpha^2\, (\theta^2\, \tau_4 + \tau_6)\, v_{k-1}^{(2)} \\
 & & \hsp \hsp \hsp + ~ 4\, \alpha^3\, \theta\, \tau_4\, v_{k-1}^{(3)} + \alpha^4\, \tau_4\, v_{k-1}^{(4)},
\end{eqnarray*}
where Table \ref{TabExp} may be read to get the coefficients appearing in the calculations. We reach the matrix formulation $V_{k} = H\, V_{k-1}$ and the initial value $V_1$ is obtained \textit{via}
\begin{eqnarray*}
v_{1}^{(0)} & = & \dE[\theta_{1}^4] ~ = ~ (\theta^4 + 6\, \theta^2\, \tau_2 + \tau_4) + 6\, \alpha^2\, \tau_2\, (\theta^2 + \tau_2) + \alpha^4\, \tau_4, \\
v_{1}^{(1)} & = & \dE[\eta_{1}\, \theta_{1}^4] ~ = ~ (4\, \theta^3\, \tau_2 + 4\, \theta\, \tau_4) + 12\, \alpha^2\, \theta\, \tau_2^2, \\
v_{1}^{(2)} & = & \dE[\eta_{1}^2\, \theta_{1}^4] ~ = ~ (\theta^4\, \tau_2 + 6\, \theta^2\, \tau_4 + \tau_6) + 6\, \alpha^2\, \tau_2\, (\theta^2\, \tau_2 + \tau_4) + \alpha^4\, \tau_2\, \tau_4, \\
v_{1}^{(3)} & = & \dE[\eta_{1}^3\, \theta_{1}^4] ~ = ~ (4\, \theta^3\, \tau_4 + 4\, \theta\, \tau_6) + 12\, \alpha^2\, \theta\, \tau_2\, \tau_4, \\
v_{1}^{(4)} & = & \dE[\eta_{1}^4\, \theta_{1}^4] ~ = ~ (\theta^4\, \tau_4 + 6\, \theta^2\, \tau_6 + \tau_8) + 6\, \alpha^2\, \tau_2\, (\theta^2\, \tau_4 + \tau_6) + \alpha^4\, \tau_4^2.
\end{eqnarray*}
Hence, $V_1 = H\, V_0$.
\end{proof}

Now for all $1 \leq k < \ell$, consider the sequence
\begin{eqnarray*}
w_{\ell,k}^{(a)} & = & \dE[\eta_{\ell}^{a}\, \theta_{\ell}^4\, \hdots\, \theta_{\ell-k+1}^4\, \theta_{\ell-k}^2\, \hdots\, \theta_1^2]
\end{eqnarray*}
where $a \in \{ 0, \hdots, 4 \}$, then build
$$
W_{\ell,k} = \begin{pmatrix} w_{\ell,k}^{(0)} \\ \vdots \\ w_{\ell,k}^{(4)} \end{pmatrix} \hspand G = \begin{pmatrix}
\theta^2 + \tau_2 & 2\, \alpha\, \theta & \alpha^2 & 0 & 0 \\
2\, \theta\, \tau_2 & 2\, \alpha\, \tau_2 & 0 & 0 & 0 \\
\theta^2\, \tau_2 + \tau_4 & 2\, \alpha\, \theta\, \tau_2 & \alpha^2\, \tau_2 & 0 & 0 \\
2\, \theta\, \tau_4 & 2\, \alpha\, \tau_4 & 0 & 0 & 0 \\
\theta^2\, \tau_4 + \tau_6 & 2\, \alpha\, \theta\, \tau_4 & \alpha^2\, \tau_4 & 0 & 0 \\
\end{pmatrix}.
$$
Once again, note that $G$ can be expressed directly from $\{ V_0, \hdots, V_4 \}$,
\begin{equation}
\label{G}
\left\{
\begin{array}{l}
G_1 = \theta^2\, V_0 + 2\, \theta\, V_1 + V_2 \\
G_2 = 2\, \alpha\, ( \theta\, V_0 + V_1) \\
G_3 = \alpha^2\, V_0 \\
G_4 = 0 \\
G_5 = 0.
\end{array}
\right.
\end{equation}
Observe also that the upper left-hand $3 \times 3$ submatrix of $G$ is precisely $M$ given by \eqref{M}. This argument will be used thereafter to establish that $\rho(G) < 1$.
\begin{lem}
\label{LemSeq22}
Assume that (H$_1$)--(H$_4$) hold. Then, for all $1 \leq k < \ell$,
\begin{equation}
\label{SystV4}
W_{\ell,k} = H^{k}\, G^{\ell-k}\, V_0.
\end{equation}
\end{lem}
\begin{proof}
The calculations are precisely the same as in the proof of Lemmas \ref{LemSeq2} and \ref{LemSeq4}. Indeed,
$$
W_{\ell,k} = H^{k}\, U_{\ell-k}
$$
where we extend the definition of $U_{k,h}$ in \eqref{Ukh} to $a \in \{ 0, \hdots, 4 \}$, namely
$$
U_{k} = \begin{pmatrix}
u_{k}^{(0)} \\
\vdots \\
u_{k}^{(4)}
\end{pmatrix} = \begin{pmatrix}
u_{k,0}^{(0)} \\
\vdots \\
u_{k,0}^{(4)}
\end{pmatrix}= U_{k,0}.
$$
Then it just remains to investigate the behavior of $u_{\ell-k}$ for $a=3$ and $a=4$ using Table \ref{TabExp},
\begin{eqnarray*}
u_{\ell-k}^{(3)} & = & \dE[\theta_{\ell-k-1}^2\, \hdots\, \theta_{1}^2\, \dE[\eta_{\ell-k}^3\, \theta_{\ell-k}^2\, \vert\, \cF_{\ell-k-1}]] \\
 & = & 2\, \theta\, \tau_4\, u_{\ell-k-1}^{(0)} + 2\, \alpha\, \tau_4\, u_{\ell-k-1}^{(1)}, \\
u_{\ell-k}^{(4)} & = & \dE[\theta_{\ell-k-1}^2\, \hdots\, \theta_{1}^2\, \dE[\eta_{\ell-k}^4\, \theta_{\ell-k}^2\, \vert\, \cF_{\ell-k-1}]] \\
 & = & (\theta^2\, \tau_4 + \tau_6)\, u_{\ell-k-1}^{(0)} + 2\, \alpha\, \theta\, \tau_4\, u_{\ell-k-1}^{(1)} + \alpha^2\, \tau_4\, u_{\ell-k-1}^{(2)}.
\end{eqnarray*}
Hence, $U_{\ell-k} = G\, U_{\ell-k-1}$. It is not hard to conclude that, for all $1 \leq k < \ell$,
$$
U_{\ell-k} = G^{\ell-k}\, V_0.
$$
\end{proof}
\begin{cor}
\label{CorMom4}
Assume that (H$_1$)--(H$_4$) hold. Then, the fourth-order properties of $(X_{t})$ are such that, for all $a \in \{ 0, \hdots, 4 \}$,
$$
\dE[\eta_{t}^{a}\, X_{t}^4] < \infty. 
$$
\end{cor}
\begin{proof}
For all $t \in \dZ$ and $k \geq 1$, denote by
\begin{equation}
\label{Mom4}
\Delta_{t} = \begin{pmatrix}
1 \\ \eta_{t} \\ \vdots \\ \eta_{t}^{4}
\end{pmatrix} \hspand P_{t,\,k} = \prod_{i=0}^{k-1} \theta_{t-i}
\end{equation}
with $P_{t,\,0} = 1$. Since $(\veps_{t})$ and $(\eta_{t})$ are uncorrelated white noises, using the causal representation \eqref{XCausal} and the same notations as above,
\begin{eqnarray*}
\dE[\Delta_{t}\, X_{t}^4] & = & \sum_{k=0}^{\infty} \sum_{\ell=0}^{\infty} \sum_{u=0}^{\infty} \sum_{v=0}^{\infty} \dE[\Delta_{t}\, P_{t,\,k}\, P_{t,\,\ell}\, P_{t,\,u}\, P_{t,\,v}\, \veps_{t-k}\, \veps_{t-\ell}\, \veps_{t-u}\, \veps_{t-v}] \\
 & = & \sigma_4\, \sum_{k=0}^{\infty} \dE[\Delta_{t}\, P_{t,\,k}^4] + 6\, \sigma_2^2 \sum_{k=0}^{\infty} \sum_{\ell=k+1}^{\infty} \dE[\Delta_{t}\, P_{t,\,k}^2\, P_{t,\,\ell}^2] \\
 & = & \sigma_4\, \sum_{k=0}^{\infty} V_{k} + 6\, \sigma_2^2 \sum_{\ell=1}^{\infty} U_{\ell} + 6\, \sigma_2^2 \sum_{k=1}^{\infty} \sum_{\ell=k+1}^{\infty} W_{\ell,k}.
\end{eqnarray*}
Then, Lemmas \ref{LemSeq4} and \ref{LemSeq22} together with the strict stationarity of $(\theta_{t})$ enable to conclude the proof under the assumptions made, since $\rho(G) = \rho(M) < 1$.
\end{proof}

We now return to the proof of Theorem \ref{ThmOLS} and we make the following notations,
\begin{eqnarray*}
M_{n}^{(1)} & = & \sum_{t=1}^{n} X_{t-1}\, \big( (1 - 2\, \alpha\, \tau_2)\, \veps_{t} + 2\, \alpha\, \theta\, \eta_{t}\, \veps_{t} + 2\, \alpha\, \eta_{t}^2\, \veps_{t} \big), \\
M_{n}^{(2)} & = & \sum_{t=1}^{n} X_{t-1}^{\, 2}\, \big( (1 - 2\, \alpha\, \tau_2 + \alpha\, \theta^2)\, \eta_{t} + \alpha\, \eta_{t}^3 + 2\, \alpha\, \theta\, (\eta_{t}^2 - \tau_2) \big), \\
M_{n}^{(3)} & = & 2\, \alpha^2 \sum_{t=1}^{n} \eta_{t-1}\, X_{t-1}\, \eta_{t}\, \veps_{t}, \\
M_{n}^{(4)} & = & \sum_{t=1}^{n} \eta_{t-1}\, X_{t-1}^{\, 2}\, \big( 2\, \alpha^2\, \theta\, \eta_{t} + 2\, \alpha^2\, (\eta_{t}^2 - \tau_2) \big), \\
M_{n}^{(5)} & = & \alpha^3 \sum_{t=1}^{n} \eta_{t-1}^2\, X_{t-1}^{\, 2}\, \eta_{t}, \\
M_{n}^{(6)} & = & \alpha \sum_{t=1}^{n} \eta_{t}\, \veps_{t}^2.
\end{eqnarray*}
Consider the filtration $(\cF^{\, *}_{n})$ given in \eqref{FiltrInf}, and let
\begin{equation}
\label{VecMart}
\cM_{n} = \begin{pmatrix}
M_{n}^{(1)} \\
\vdots \\
M_{n}^{(6)}
\end{pmatrix}.
\end{equation}
Under our hypotheses, $\cM_{n}$ is a locally square-integrable real vector $(\cF^{\, *}_{n})$--martingale. Once again we will make use of the central limit theorem for vector martingales, as in the proof of Theorem \eqref{ThmEmpMean}. On the one hand, we have to study the asymptotic behavior of the predictable
quadratic variation of $\cM_{n}$. For all $n \geq 1$, let
\begin{equation}
\label{MartQuad}
\langle \cM \rangle_{n} = \sum_{t=1}^{n} \dE\big[ (\Delta \cM_{t} )(\Delta \cM_{t} )^{T}\, \vert\, \cF^{\, *}_{t-1} \big],
\end{equation}
with $\Delta \cM_1 = \cM_1$. To simplify the calculations, we introduce some more notations. The second-order moments of the process are defined in \eqref{Lam} and its fourth-order moments are called
\begin{equation}
\label{Del}
\dE[\Delta_{t}\, X_{t}^4] = \begin{pmatrix}
\delta_{0} \\
\vdots \\
\delta_{4}
\end{pmatrix} = \Delta
\end{equation}
where $\Delta_{t}$ is given in \eqref{Mom4}. We use the strict stationarity to establish, following Corollaries \ref{CorMom2} and \ref{CorMom4}, that
\begin{equation}
\label{DelExpl}
\Delta = (I_5-H)^{-1}\, (\sigma_2\, R + \sigma_4\, V_0)
\end{equation}
in which $R$ is defined from \eqref{G} as
$$
R = 6\, \lambda_0\, G_1 + 6\, \lambda_1\, G_2 + 6\, \lambda_2\, G_3.
$$
Now, we are going to show that the asymptotic behavior of $\langle \cM \rangle_{n}/n$ is entirely described by $\Lambda$ and $\Delta$. By ergodicity,
\begin{equation}
\label{CvgEmpMom4}
\frac{1}{n} \sum_{t=1}^{n} \Delta_{t}\, X_{t}^4 \cvgas \Delta.
\end{equation}
We get back to \eqref{MartQuad}. First, there exists constants such that
\begin{eqnarray*}
\langle M^{(1)}, M^{(2)} \rangle_{n} & = & \sum_{t=1}^{n} X_{t-1}^3\, \dE\big[ \big( k_{(1)} + k_{(2)}\, \eta_{t} + k_{(3)}\, \eta_{t}^2 \big) \big( k_{(4)}\, \eta_{t} + k_{(5)}\, \eta_{t}^3 \\
 & & \hsp \hsp \hsp + ~  k_{(6)}\, (\eta_{t}^2 - \tau_2) \big) \veps_{t} \big] ~ = ~ 0 \end{eqnarray*}
under our assumptions. \textit{Via} analogous arguments, it follows that
\begin{eqnarray*}
\langle M^{(1)}, M^{(4)} \rangle_{n} & = & \langle M^{(1)}, M^{(5)} \rangle_{n} ~ = ~ \langle M^{(1)}, M^{(6)} \rangle_{n} ~ = ~ \langle M^{(2)}, M^{(3)} \rangle_{n} \\
 & = & \langle M^{(3)}, M^{(4)} \rangle_{n} ~ = ~ \langle M^{(3)}, M^{(5)} \rangle_{n} ~ = ~ \langle M^{(3)}, M^{(6)} \rangle_{n} ~ = ~ 0.
\end{eqnarray*}
Then we look at nonzero contributions, where we use the constants defined in \eqref{CstMartQuad} and \eqref{K}. Since the calculations are very similar we only detail the first one,
\begin{eqnarray*}
\langle M^{(1)} \rangle_{n} & = & \sum_{t=1}^{n} X_{t-1}^{\, 2}\, \dE\big[ \big( (1 - 2\, \alpha\, \tau_2)\, \veps_{t} + 2\, \alpha\, \theta\, \eta_{t}\, \veps_{t} + 2\, \alpha\, \eta_{t}^2\, \veps_{t} \big)^2 \big] \\
 & = &  \sigma_2\, \big( 1 + 4\, \alpha^2\, ( \theta^2\, \tau_2 - \tau_2^2 + \tau_4) \big) \sum_{t=1}^{n} X_{t-1}^{\, 2}.
\end{eqnarray*}
To sum up, we obtain
\begin{equation}
\label{MartQuadExpl}
\langle \cM \rangle_{n} = K \circ \sum_{t=1}^{n} \begin{pmatrix}
X_{t}^2 & 0 & \eta_{t}\, X_{t}^2 & 0 & 0 & 0 \\
0 & X_{t}^4 & 0 & \eta_{t}\, X_{t}^4 & \eta_{t}^2\, X_{t}^4 & X_{t}^2 \\
\eta_{t}\, X_{t}^2 & 0 & \eta_{t}^2\, X_{t}^2 & 0 & 0 & 0 \\
0 & \eta_{t}\, X_{t}^4 & 0 & \eta_{t}^2\, X_{t}^4 & \eta_{t}^3\, X_{t}^4 & \eta_{t}\, X_{t}^2 \\
0 & \eta_{t}^2\, X_{t}^4 & 0 & \eta_{t}^3\, X_{t}^4 & \eta_{t}^4\, X_{t}^4 & \eta_{t}^2\, X_{t}^2 \\
0 & X_{t}^2 & 0 & \eta_{t}\, X_{t}^2 & \eta_{t}^2\, X_{t}^2 & 1
\end{pmatrix} + R_{n}
\end{equation}
where the Hadamard product $\circ$ is used to lighten the formula, and where the remainder $R_{n}$ is made of isolated terms such that
\begin{equation}
\label{MartQuadCvgRem}
\frac{R_{n}}{n} \cvgas 0.
\end{equation}
To reach these results, we refer the reader to the explanations following \eqref{MartQuadCvgRemEmpMean} since the same methodology has just been applied on $\cM_{n}$. The combination of \eqref{CvgEmpMom2}, \eqref{CvgEmpMom4}, \eqref{MartQuadExpl} and \eqref{MartQuadCvgRem} leads to
\begin{equation}
\label{MartQuadCvg}
\frac{\langle \cM \rangle_{n}}{n} \cvgas K \circ \Gamma
\end{equation}
where $\Gamma$ is given by
\begin{equation}
\label{Gam}
\Gamma = \begin{pmatrix}
\lambda_0 & 0 & \lambda_1 & 0 & 0 & 0 \\
0 & \delta_0 & 0 & \delta_1 & \delta_2 & \lambda_0 \\
\lambda_1 & 0 & \lambda_2 & 0 & 0 & 0 \\
0 & \delta_1 & 0 & \delta_2 & \delta_3 & \lambda_1 \\
0 & \delta_2 & 0 & \delta_3 & \delta_4 & \lambda_2 \\
0 & \lambda_0 & 0 & \lambda_1 & \lambda_2 & 1
\end{pmatrix}.
\end{equation}
On the other hand, it is necessary to prove that the Lindeberg's condition is satisfied, namely that for all $\veps > 0$,
\begin{equation}
\label{MartLind}
\frac{1}{n} \sum_{t=1}^{n} \dE\big[ \Vert \Delta \cM_{t} \Vert^2\, \dI_{ \{ \Vert \Delta \cM_{t} \Vert\, \geq\, \veps\, \sqrt{n} \}}\, \vert\, \cF^{\, *}_{t-1} \big] \cvgp 0
\end{equation}
as $n$ tends to infinity. The result follows from Corollaries \ref{CorMom2} and \ref{CorMom4}, together with the same reasoning as the one used to establih \eqref{MartLindEmpMean}. From \eqref{MartQuadCvg} and \eqref{MartLind}, we deduce that
\begin{equation}
\label{MartCvg}
\frac{\cM_{n}}{\sqrt{n}} \cvgl \cN(0, K \circ \Gamma).
\end{equation}
Finally, by a very tedious but straightforward calculation, one can obtain that
\begin{equation}
\label{DecompMart}
\sqrt{n}\, \big( \wh{\theta}_{n} - \theta^{*} \big) = \frac{n}{\sum_{t=1}^{n} X_{t-1}^{\, 2}}\, \frac{\Omega_{6}^{\, T} \cM_{n} + r_{n}}{(1 - 2\, \alpha\, \tau_2)\, \sqrt{n}}
\end{equation}
where $\Omega_{6}^{\, T} = (1 ~ 1 ~ 1 ~ 1 ~ 1 ~ 1)$ and $r_{n} = o(\sqrt{n})$ a.s. from \eqref{CvgEmpMom2} and \eqref{CvgEmpMom4}. It remains to apply Slutsky's lemma to conclude that
$$
\sqrt{n}\, \big( \wh{\theta}_{n} - \theta^{*} \big) \cvgl \cN(0, \omega^2)
$$
with
\begin{equation}
\label{VarOmega}
\omega^2 = \frac{\Omega_{6}^{\, T} (K \circ \Gamma)\, \Omega_{6}}{ \lambda_0^2\, (1 - 2\, \alpha\, \tau_2)^2}
\end{equation}
using the whole notations above. $\hfill\qed$

\subsection{Proof of Theorem \ref{ThmOLSRat}}
\label{SecProofThmOLSRat}

Letting $V_{n} = \sqrt{n}\, I_6$, such a sequence obviously satisfies the regular growth conditions of \cite{ChaabaneMaaouia00}. Keeping the notations of \eqref{VecMart}, we have studied the hook of $\cM_{n}$ in \eqref{MartQuadCvg} and Lindeberg's condition is already fulfilled in \eqref{MartLind}, it only remains to check that
\begin{equation}
\label{MartDiffVar}
\frac{[\cM]_{n} - \langle \cM \rangle_{n}}{n} \cvgas 0
\end{equation}
where
\begin{equation*}
[\cM]_{n} = \sum_{t=1}^{n} (\Delta \cM_{t}) (\Delta \cM_{t})^{\, T}
\end{equation*}
is the total variation of $\cM_{n}$, to apply Thm. 2.1 of \cite{ChaabaneMaaouia00}. To be precise with the required hypotheses, note that \eqref{MartLind} also holds almost surely, by ergodicity. But \eqref{MartDiffVar} is an immediate consequence of the ergodicity of the increments. Thus,
\begin{equation*}
\frac{1}{6\, \ln n} \sum_{t=1}^{n} \left[ 1 - \left( \frac{t}{t+1} \right)^{\! 6} \right] \frac{\cM_{t}\, \cM_{t}^{\, T}}{t} \cvgas K \circ \Gamma
\end{equation*}
and, after simplifications,
\begin{equation}
\label{LfqMart}
\frac{1}{\ln n} \sum_{t=1}^{n} \frac{\cM_{t}\, \cM_{t}^{\, T}}{t^{\, 2}} \cvgas K \circ \Gamma.
\end{equation}
The remainder $r_{n}$ in \eqref{DecompMart} is a long linear combination of isolated terms, we detail here the treatment of the largest one which takes the form of $\eta_{n-1}^2\, X_{n-1}^2\, \eta_{n}$. Corollary \ref{CorMom4} implies, for $a=4$ and \textit{via} the ergodic theorem, that
\begin{equation*}
\frac{1}{n} \sum_{t=1}^{n} \eta_{t-1}^4\, X_{t-1}^4\, \eta_{t}^2 \cvgas \delta_4\, \tau_2,
\end{equation*}
which in turn leads to
\begin{equation*}
\frac{\eta_{n-1}^4\, X_{n-1}^4\, \eta_{n}^2}{n} \cvgas 0 \hsp \text{so that} \hsp \frac{\eta_{n-1}^4\, X_{n-1}^4\, \eta_{n}^2}{n^2} = o(n^{-1}) \hspp \textnormal{a.s.}
\end{equation*}
It follows that
\begin{equation*}
\sum_{t=1}^{n} \frac{\eta_{t-1}^4\, X_{t-1}^4\, \eta_{t}^2}{t^2} = o\bigg( \sum_{t=1}^{n} \frac{1}{t} \bigg) = o(\ln n) \hspp \textnormal{a.s.}
\end{equation*}
By extrapolation, treating similarly all residual terms,
\begin{equation}
\label{LfqRn}
\frac{1}{\ln n} \sum_{t=1}^{n} \frac{r_{t}^{\, 2}}{t^{\, 2}} \cvgas 0.
\end{equation}
It remains to combine these results to get
\begin{eqnarray*}
\frac{(1 - 2\, \alpha\, \tau_2)^2}{\ln n} \sum_{t=1}^{n} \big( \wh{\theta}_{t} - \theta^{*} \big)^2 & = & \frac{1}{\ln n} \sum_{t=1}^{n} \frac{\Omega_{6}^{\, T} \cM_{t}\, \cM_{t}^{\, T} \Omega_{6}}{S_{t-1}^{\, 2}} + \frac{1}{\ln n} \sum_{t=1}^{n} \frac{r_{t}^{\, 2}}{S_{t-1}^{\, 2}} \\
 & & \hsp \hsp \hsp + ~ \frac{2}{\ln n} \sum_{t=1}^{n}  \frac{\Omega_{6}^{\, T} \cM_{t}\, r_{t}}{S_{t-1}^{\, 2}}
\end{eqnarray*}
where
\begin{equation}
\label{Sn}
S_{n} = \sum_{t=0}^{n} X_{t}^{\, 2} \hsp \text{satisfies} \hsp \frac{S_{n}}{n} \cvgas \lambda_0.
\end{equation}
Using Cauchy-Schwarz inequality, the cross-term is shown to be negligible. From \eqref{VarOmega}, \eqref{LfqMart}, \eqref{LfqRn} and the previous remark,
\begin{equation*}
\frac{1}{\ln n} \sum_{t=1}^{n} \big( \wh{\theta}_{t} - \theta^{*} \big)^2 \cvgas \frac{\Omega_{6}^{\, T} (K \circ \Gamma)\, \Omega_{6}}{\lambda_0^2\, (1 - 2\, \alpha\, \tau_2)^2} = \omega^2
\end{equation*}
which concludes the first part of the proof and follows from Toeplitz lemma applied in the right-hand side of the decomposition. The rate of convergence of $\wh{\theta}_{n}$ is easier to handle. As a matter of fact, we have already seen that $\cM_{n}$ is a vector $(\cF^{\, *}_{n})$--martingale having ergodic and stationary increments. So,
\begin{equation}
\label{MartCL}
\cN_{n} = \Omega_{6}^{\, T} \cM_{n}
\end{equation}
is a scalar $(\cF^{\, *}_{n})$--martingale having the same incremental properties, and our hypotheses guarantee that $\dE[(\Delta\, \cN_1)^2] = \Omega_{6}^{\, T} (K \circ \Gamma)\, \Omega_{6} < \infty$. The main theorem of \cite{Stout70} enables to infer that
\begin{equation}
\label{LliMartSup}
\limsup_{n\, \rightarrow\, +\infty}~ \frac{\cN_{n}}{\sqrt{2\, n\, \ln \ln n}} = \sqrt{\Omega_{6}^{\, T} (K \circ \Gamma)\, \Omega_{6}} \hspp \textnormal{a.s.}
\end{equation}
and
\begin{equation}
\label{LliMartInf}
\liminf_{n\, \rightarrow\, +\infty}~ \frac{\cN_{n}}{\sqrt{2\, n\, \ln \ln n}} = -\sqrt{\Omega_{6}^{\, T} (K \circ \Gamma)\, \Omega_{6}} \hspp \textnormal{a.s.}
\end{equation}
replacing $\cN_{n}$ by $-\cN_{n}$. Thus, once again exploiting \eqref{DecompMart},
\begin{eqnarray*}
\limsup_{n\, \rightarrow\, +\infty}~ \sqrt{\frac{n}{2\, \ln \ln n}}\, \big( \wh{\theta}_{n} - \theta^{*} \big) & = & \frac{1}{\lambda_0\, (1 - 2\, \alpha\, \tau_2)}~ \limsup_{n\, \rightarrow\, +\infty}~ \frac{\cN_{n} + r_{n}}{\sqrt{2\, n\, \ln \ln n}} \\
 & = & \omega \hspp \textnormal{a.s.}
\end{eqnarray*}
using \eqref{LliMartSup} and the fact that $r_{n} = o(\sqrt{n})$ a.s. The symmetric result is reached from \eqref{LliMartInf} and the proof is complete. $\hfill\qed$

\subsection{Proof of Theorem \ref{ThmTlcNewEst}}
\label{SecProofThmTlcNewEst}

One shall prove this result in two steps. First, we will identify the covariance $\Sigma$ such that
\begin{equation}
\label{NormCouple}
\sqrt{n}\, \begin{pmatrix}
\wh{\theta}_{n} - \theta^{*} \\
\wh{\vartheta}_{n} - \vartheta^{*}
\end{pmatrix} \cvgl \cN(0, \Sigma)
\end{equation}
where $\wh{\theta}_{n}$ and $\wh{\vartheta}_{n}$ are given in \eqref{OLSEst2}, $\theta^{*} = \rho_{X}(1)$ is the limiting value of $\wh{\theta}_{n}$ deeply investigated up to this point and
\begin{equation*}
\vartheta^{*} = \rho_{X}(2) = \frac{\theta^2 + \alpha\, \tau_2\, (1 - 2\, \alpha\, \tau_2)}{1 - 2\, \alpha\, \tau_2}.
\end{equation*}
Then we will translate the result to the new estimates \eqref{NewEst} \textit{via} the Delta method. Of course the first step being very close to the proof of Theorem \ref{ThmOLS}, we only give an outline of the calculations. The second-order lag in $\wh{\vartheta}_{n}$ gives a new scalar $(\cF^{\, *}_{n})$--martingale contribution that we will define as
\begin{eqnarray}
\label{NewMart}
\cL_{n} & = & \alpha \sum_{t=1}^{n} X_{t-1}\, \eta_{t}\, \veps_{t} + \sum_{t=1}^{n} X_{t-1}^{\, 2}\, \big( \alpha\, \theta\, \eta_{t} + \alpha\, (\eta_{t}^2 - \tau_2) \big) \nonumber \\
 & & \hsp \hsp \hsp + ~ \alpha^2 \sum_{t=1}^{n} \eta_{t-1}\, X_{t-1}^{\, 2}\, \eta_{t} + \sum_{t=2}^{n} X_{t-2}\, \veps_{t} + \sum_{t=2}^{n} X_{t-2}\, \veps_{t-1}\, \eta_{t} \nonumber \\
  & & \hsp \hsp \hsp + ~ \theta \sum_{t=2}^{n} X_{t-2}^{\, 2}\, \eta_{t} + \sum_{t=2}^{n} X_{t-2}^{\, 2}\, \eta_{t-1}\, \eta_{t} + \alpha \sum_{t=2}^{n} \eta_{t-2}\, X_{t-2}^{\, 2}\, \eta_{t}
\end{eqnarray}
which follows from a very tedious development of $\sum_{t=2}^{n} X_{t-2} X_{t}.$ An exhaustive expansion of $\wh{\vartheta}_{n} - \vartheta^{*}$ leads to
\begin{equation*}
\big( \wh{\vartheta}_{n} - \vartheta^{*} \big) S_{n-2} = \theta^{*}\, \Omega_{6}^{\, T} \cM_{n} + \cL_{n} + s_{n}
\end{equation*}
where $\cM_{n}$ is given in \eqref{VecMart}, $S_{n}$ in \eqref{Sn}, $\Omega_{6}^{\, T} = (1 ~ 1 ~ 1 ~ 1 ~ 1 ~ 1)$ and $s_{n}$ is made of isolated terms, each one being $o(\sqrt{n})$ a.s. as soon as the process has fourth-order moments, \textit{i.e.} under (H$_4$). Combined with \eqref{DecompMart},
\begin{equation}
\label{DecompNewMart}
\sqrt{n}\, \begin{pmatrix}
\wh{\theta}_{n} - \theta^{*} \\
\wh{\vartheta}_{n} - \vartheta^{*}
\end{pmatrix} = \frac{A_{n}}{\sqrt{n}}\, \begin{pmatrix}
\cM_{n} \\ \cL_{n}
\end{pmatrix} + T_{n}
\end{equation}
where
\begin{equation}
\label{An}
A_{n} = \begin{pmatrix}
\frac{n}{S_{n-1}}\, \frac{\Omega_{6}^{\, T}}{1 - 2\, \alpha\, \tau_2} & 0 \\
\frac{n}{S_{n-2}}\, \frac{\theta\, \Omega_{6}^{\, T}}{1 - 2\, \alpha\, \tau_2} & \frac{n}{S_{n-2}}
\end{pmatrix} \cvgas A = \begin{pmatrix}
\frac{\Omega_{6}^{\, T}}{\lambda_0\, (1 - 2\, \alpha\, \tau_2)} & 0 \\
\frac{\theta\, \Omega_{6}^{\, T}}{\lambda_0\, (1 - 2\, \alpha\, \tau_2)} & \frac{1}{\lambda_0}
\end{pmatrix}
\end{equation}
are matrices of size $2 \times 7$ and $T_{n} = o(1)$ a.s. We have to study the hook of this new vector $(\cF^{\, *}_{n})$--martingale. First, $\langle \cM \rangle_{n}$ is already treated in \eqref{MartQuadCvg}. For the cross-term and the last one, we need more notations. Let
\begin{equation}
\label{Mu}
\mu_{a,b,c,p,q} = \dE[\eta_{t-1}^{\, a}\, \eta_{t}^{\, b}\, \veps_{t}^{\, c}\, X_{t-1}^{\, p}\, X_{t}^{\, q}]
\end{equation}
and observe that $\mu_{0,b,0,0,2} = [\Lambda]_{b+1}$ in \eqref{Lam} for $b \in \{ 0,1,2 \}$ and that $\mu_{0,b,0,0,4} = [\Delta]_{b+1}$ in \eqref{Del} for $b \in \{ 0, \hdots, 4 \}$. Then, it can be seen \textit{via} analogous arguments as usual relying on ergodicity and negligible isolated terms, that
\begin{equation}
\label{NewMartQuadMLCvg}
\frac{\langle \cM, \cL \rangle_{n}}{n} \cvgas (L \circ \Upsilon)\, \Omega_{6}
\end{equation}
where $L$ is defined in \eqref{L} and $\Upsilon$ is given by
\begin{equation}
\label{Ups}
\Upsilon = \begin{pmatrix}
\theta^{*}\, \lambda_0 & \lambda_0 & 0 & 0 & 0 & 0 \\
\delta_0 & \delta_1 & \mu_{0,0,0,2,2} & \mu_{0,1,0,2,2} & \mu_{1,0,0,2,2} & \mu_{0,0,1,1,2} \\
\lambda_1 & 0 & 0 & 0 & 0 & 0 \\
\delta_1 & \delta_2 & \mu_{0,1,0,2,2} & \mu_{0,2,0,2,2} & \mu_{1,1,0,2,2} & \mu_{0,1,1,1,2} \\
\delta_2 & \delta_3 & \mu_{0,2,0,2,2} & \mu_{0,3,0,2,2} & \mu_{1,2,0,2,2} & \mu_{0,2,1,1,2} \\
\lambda_0 & \lambda_1 & 0 & 0 & 0 & 0
\end{pmatrix}.
\end{equation}
Finally, we have
\begin{eqnarray}
\label{NewMartQuadLCvg}
\frac{\langle \cL \rangle_{n}}{n} & \cvgas & \ell ~ = ~ m_{(1)}\, \lambda_0 + m_{(2)}\, \delta_0 + m_{(3)}\, \delta_1 + m_{(4)}\, \delta_2 + \theta\, m_{(5)}\, \mu_{0,0,0,2,2} \nonumber \\
 & & \hsp \hsp \hsp + ~ \alpha\, m_{(5)}\, \mu_{1,0,0,2,2} + (1 + \alpha)\, m_{(5)}\, \mu_{0,1,0,2,2} + m_{(5)}\, \mu_{0,0,1,1,2} \nonumber \\
 & & \hsp \hsp \hsp + ~ m_{(6)}\, \mu_{0,2,0,2,2} + \alpha\, m_{(6)}\, \mu_{1,1,0,2,2} + m_{(6)}\, \mu_{0,1,1,1,2}
\end{eqnarray}
where the constants are detailed in \eqref{CstNewMartQuadL}. This last convergence, together with \eqref{NewMartQuadMLCvg}, \eqref{MartQuadCvg} and their related notations, implies
\begin{equation}
\label{NewMartQuadCvg}
\frac{1}{n} \left\langle \begin{matrix} \begin{pmatrix} \cM \\ \cL \end{pmatrix} \end{matrix} \right\rangle_{\! n} \cvgas \Sigma_{\textnormal{ML}} = \begin{pmatrix}
K \circ \Gamma & (L \circ \Upsilon)\, \Omega_{6} \\
\Omega_{6}^{\, T} (L \circ \Upsilon)^{\, T}\, & \ell
\end{pmatrix}. 
\end{equation}
Lindeberg's condition is clearly fulfilled and Slutsky's lemma applied on the relation \eqref{DecompNewMart}, taking into account the asymptotic normality of the martingale and the remarks that follow \eqref{DecompNewMart}, enables to identify $\Sigma$ in \eqref{NormCouple} as
\begin{equation}
\label{Sig}
\Sigma = A\, \Sigma_{\textnormal{ML}} A^{\, T}
\end{equation}
where $A$ is given in \eqref{An}. This ends the first part of the proof.

\begin{rem}
\label{RemSig}
It is important to note that, despite the complex structure of $\Sigma$, it only depends on the parameters and can be computed explicitely. Indeed, it is easy to see that all coefficients $\mu_{a,b,c,p,q}$ in $\Sigma_{\textnormal{ML}}$ exist under our hypotheses, exploiting the fourth-order moments of the process. We can compute each of them using the same lines as in our previous technical lemmas.
\end{rem}

\noindent Consider now the mapping $f$ in \eqref{MapDelta} whose Jacobian matrix is
\begin{equation*}
\nabla f(x,y) = \begin{pmatrix}
\frac{(1 - 2 y)\, (1 + 2 x^2)}{(1 - 2 x^2)^2} & \frac{-2 x}{1 - 2 x^2} \\
\frac{-2 x\, (1 - 2 y)}{(1 - 2 x^2)^2} & \frac{1}{1 - 2 x^2}
\end{pmatrix}.
\end{equation*}
The couple of estimates \eqref{NewEst} therefore satisfies
\begin{equation*}
\sqrt{n}\, \begin{pmatrix}
\wt{\theta}_{n} - \theta \\
\wt{\gamma}_{n} - \gamma
\end{pmatrix} \cvgl \cN(0, \nabla^{\, T} f(\theta^{*}, \vartheta^{*})\, \Sigma\, \nabla f(\theta^{*}, \vartheta^{*}))
\end{equation*}
by application of the Delta method, the pathological cases $\theta^{*} = \pm\frac{1}{\sqrt{2}}$ being excluded from the study. $\hfill\qed$

%%%%%%%%%%%%%%%%%%%%%%%%%%%%%%%%%%%%%%%%%%%%%%%%%%%%%%%%%%%%%%%%%%%%%%%%%%%%%%%%

%%%%%%%%%%%%%%%%%%%%%%%%%%%%%%%%%%%%%%%%%%%%%%%%%%%%%%%%%%%%%%%%%%%%%%%%%%%%%%%%
\section*{Appendix}
\label{SecApp}

\setcounter{equation}{0}
\renewcommand{\theequation}{A.\arabic{equation}}

This appendix is devoted to the numerous constants of the study, for greater clarity. The first of them are given by
\begin{equation}
\label{CstMartEmpMean}
\left\{
\begin{array}{lcl}
\bar{k}_{(1)} & = & (1 + \alpha\, \theta)^2\, \tau_2 + \alpha^2\, (\tau_4 - \tau_2^2) \\
\bar{k}_{(1-2)} & = & \alpha^2\, (1 + \alpha\, \theta)\, \tau_2 \\
\bar{k}_{(2)} & = & \alpha^4\, \tau_2 \\
\bar{k}_{(3)} & = & (1 + \alpha^2\, \tau_2)\, \sigma_2 \\
\end{array}
\right.
\end{equation}
and serve to build the matrix
\begin{equation}
\label{Kb}
\bar{K} = \begin{pmatrix}
\bar{k}_{(1)} & \bar{k}_{(1-2)} & 0 \\
\bar{k}_{(1-2)} & \bar{k}_{(2)} & 0 \\
0 & 0 & \bar{k}_{(3)}
\end{pmatrix}.
\end{equation}
We also define
\begin{equation}
\label{CstMartQuad}
\left\{
\begin{array}{lcl}
k_{(1)} & = & \sigma_2\, ( 1 + 4\, \alpha^2\, ( \theta^2\, \tau_2 - \tau_2^2 + \tau_4) ) \\
k_{(1-3)} & = & 4\, \alpha^3\, \theta\, \tau_2\, \sigma_2 \\
k_{(2)} & = & (1 - 2\, \alpha\, \tau_2 + \alpha\, \theta^2)\, ( 2\, \alpha\, \tau_4 + \tau_2\, (1 - 2\, \alpha\, \tau_2 + \alpha\, \theta^2)) \\
 & & \hsp \hsp \hsp + ~ \alpha^2\, ( \tau_6 + 4\, \theta^2\, (\tau_4 - \tau_2^2) ) \\
k_{(2-4)} & = & 2\, \alpha^2\, \theta\, \tau_2\, (1 + \alpha\, \theta^2 - 4\, \alpha\, \tau_2) + 6\, \alpha^3\, \theta\, \tau_4 \\
k_{(2-5)} & = & \alpha^3\, (\alpha\, \tau_4 + \tau_2\, (1 - 2\, \alpha\, \tau_2 + \alpha\, \theta^2)) \\
k_{(2-6)} & = & \alpha\, \sigma_2\, ( \alpha\, \tau_4 + \tau_2\, (1 - 2\, \alpha\, \tau_2 + \alpha\, \theta^2) ) \\
k_{(3)} & = & 4\, \alpha^4\, \tau_2\, \sigma_2 \\
k_{(4)} & = & 4\, \alpha^4\, ( \theta^2\, \tau_2 - \tau_2^2 + \tau_4) \\
k_{(4-5)} & = & 2\, \alpha^5\, \theta\, \tau_2 \\
k_{(4-6)} & = & 2\, \alpha^3\, \theta\, \tau_2\, \sigma_2 \\
k_{(5)} & = & \alpha^6\, \tau_2 \\
k_{(5-6)} & = & \alpha^4\, \tau_2\, \sigma_2 \\
k_{(6)} & = & \alpha^2\, \tau_2\, \sigma_4
\end{array}
\right.
\end{equation}
that we put in the matrix form
\begin{equation}
\label{K}
K = \begin{pmatrix}
k_{(1)} & 0 & k_{(1-3)} & 0 & 0 & 0 \\
0 & k_{(2)} & 0 & k_{(2-4)} & k_{(2-5)} & k_{(2-6)} \\
k_{(1-3)} & 0 & k_{(3)} & 0 & 0 & 0 \\
0 & k_{(2-4)} & 0 & k_{(4)} & k_{(4-5)} & k_{(4-6)} \\
0 & k_{(2-5)} & 0 & k_{(4-5)} & k_{(5)} & k_{(5-6)} \\
0 & k_{(2-6)} & 0 & k_{(4-6)} & k_{(5-6)} & k_{(6)}
\end{pmatrix}.
\end{equation}
Moreover, we have to consider
\begin{equation}
\label{CstNewMartQuadML}
\left\{
\begin{array}{lcl}
\ell_{(1)}^{\, \prime} & = & \sigma_2 \\
\ell_{(1)} & = & 2\, \alpha^2\, \theta\, \tau_2\, \sigma_2 \\
\ell_{(2)}^{\, \prime} & = & \alpha\, \theta\, (\tau_2\, (1 - 2\, \alpha\, \tau_2 + \alpha\, \theta^2) - \alpha\, (2\, \tau_2^2 - 3\, \tau_4)) \\
\ell_{(2)} & = & \alpha\, \tau_4 + \tau_2\, (1 - 2\, \alpha\, \tau_2 + \alpha\, \theta^2) \\
\ell_{(3)} & = & 2\, \alpha^3\, \tau_2\, \sigma_2 \\
\ell_{(4)}^{\, \prime} & = & 2\, \alpha^3\, ( \theta^2\, \tau_2 - \tau_2^2 + \tau_4) \\
\ell_{(4)} & = & 2\, \alpha^2\, \theta\, \tau_2 \\
\ell_{(5)} & = & \alpha^4\, \tau_2 \\
\ell_{(6)} & = & \alpha\, \tau_2\, \sigma_2\, (1+\alpha)
\end{array}
\right.
\end{equation}
in the matrix form
\begin{equation}
\label{L}
L = \begin{pmatrix}
\ell_{(1)}^{\, \prime} & \ell_{(1)} & 0 & 0 & 0 & 0 \\
\ell_{(2)}^{\, \prime} & \alpha^2\, \ell_{(2)} & \theta\, \ell_{(2)} & \ell_{(2)} & \alpha\, \ell_{(2)} & \ell_{(2)} \\
\ell_{(3)} & 0 & 0 & 0 & 0 & 0 \\
\ell_{(4)}^{\, \prime} & \alpha^2\, \ell_{(4)} & \theta\, \ell_{(4)} & \ell_{(4)} & \alpha\, \ell_{(4)} & \ell_{(4)} \\
\alpha\, \theta\, \ell_{(5)} & \alpha^2\, \ell_{(5)} & \theta\, \ell_{(5)} & \ell_{(5)} & \alpha\, \ell_{(5)} & \ell_{(5)} \\
\theta\, \ell_{(6)} & \alpha\, \ell_{(6)} & 0 & 0 & 0 & 0
\end{pmatrix}.
\end{equation}
We conclude by a last set of constants,
\begin{equation}
\label{CstNewMartQuadL}
\left\{
\begin{array}{lcl}
m_{(1)} & = & \sigma_2\, (1 + \tau_2\, (1 + \alpha^2)) \\
m_{(2)} & = & \, \theta^2\, (1 + \alpha^2)\, \tau_2 + (1 - \alpha^2)\, \tau_2^2 + \alpha^2\, \tau_4 \\
m_{(3)} & = & 2\, \alpha\, \theta\, (1 + \alpha^2)\, \tau_2 \\
m_{(4)} & = & \alpha^2\, (1 + \alpha^2)\, \tau_2 \\
m_{(5)} & = & 2\, \alpha\, \theta\, \tau_2 \\
m_{(6)} & = & 2\, \alpha^2\, \tau_2.
\end{array}
\right.
\end{equation}

%%%%%%%%%%%%%%%%%%%%%%%%%%%%%%%%%%%%%%%%%%%%%%%%%%%%%%%%%%%%%%%%%%%%%%%%%%%%%%%%
\nocite{*}

\bibliographystyle{plain}
\bibliography{RCARMACoef}

\vspace{10pt}

\end{document}